\documentclass[12pt]{article}

\usepackage{amssymb, amsmath,amsthm,mathtools,mathrsfs,color}
\usepackage[all]{xy}
\usepackage{enumerate}
\usepackage{enumitem}   
\usepackage{booktabs}
\usepackage{caption}

\usepackage{stmaryrd} 
\usepackage{multirow,longtable}
\usepackage{pgffor}
\usepackage{tikz,tikz-cd}
\usepackage{hyperref}
\hypersetup{
    colorlinks,
    linkcolor=blue, 
    citecolor={blue},
    urlcolor={blue}
}
\usepackage{geometry}
\setlist{itemsep=1pt,topsep=5pt,parsep=0pt} 
\setlist[enumerate]{label={\rm (\arabic{*})}}

\let\C\relax  

\definecolor{shadecolor}{RGB}{241, 241, 255}

\usetikzlibrary{calc}
\tikzset{
	dot/.style={draw,circle,inner sep=0.8pt,fill=black},
	cir/.style={draw,circle,inner sep=0.8pt}
}
\DeclareMathAlphabet{\mathcal}{OMS}{cmsy}{m}{n}

\theoremstyle{plain}
\newtheorem{Def}{Definition}[section]
\newtheorem{Prop}[Def]{Proposition}
\newtheorem{Theo}[Def]{Theorem}
\newtheorem{Lem}[Def]{Lemma}
\newtheorem{Coro}[Def]{Corollary}

\newtheorem{Remark}[Def]{Remark}
\theoremstyle{definition}
\newtheorem{Example}[Def]{Example}
\newtheorem{Eg}[Def]{Example}

\DeclareMathOperator{\add}{add}

\DeclareMathOperator{\gldim}{gldim}
\DeclareMathOperator{\Ext}{Ext}

\DeclareMathOperator{\Aut}{Aut}

\DeclareMathOperator{\End}{End}

\DeclareMathOperator{\Hom}{Hom}

\DeclareMathOperator{\supp}{\sf supp}

\DeclareMathOperator{\Ker}{Ker}

\DeclareMathOperator{\ind}{\sf ind}
\DeclareMathOperator{\id}{\sf id}
\DeclareMathOperator{\domdim}{dom.dim}
\DeclareMathOperator{\rigdim}{rig.dim}
\DeclareMathOperator{\rd}{rd}

\newcommand{\defCategory}[2]{
	\newcommand{#1}{#2\defvariable}}

\NewDocumentCommand{\defvariable}{O{}O{}}{
	\if\relax\detokenize{#2}\relax 
	\if\relax\detokenize{#1}\relax 
	\else 
	({#1})
	\fi
	\else
	\if\relax\detokenize{#1}\relax
	^{{\rm #2}}
	\else 
	^{{\rm #2}}({#1}) 
	\fi 
	\fi } 

\defCategory{\C}{\mathcal{C}}
\defCategory{\K}{\mathcal{K}}
\defCategory{\D}{\mathcal{D}}
\defCategory{\sA}{\mathcal{A}}
\defCategory{\sB}{\mathcal{B}}
\defCategory{\sC}{\mathscr{C}}
\defCategory{\sD}{\mathscr{D}}

\newcommand{\Db}[1][]{\D[#1][b]}

\newcommand{\T}{\mathcal{T}}

\def\modcat#1{{\sf mod}\mbox{-}{#1}}

\def\Modcat#1{{\sf Mod}\mbox{-}{#1}}

\def\stmodcat#1{\underline{\sf mod}\mbox{-}{#1}}

\newcommand{\lra}{\longrightarrow}

\newcommand{\lraf}[1]{\stackrel{#1}{\lra}}

\renewcommand{\leq}{\leqslant}
\renewcommand{\geq}{\geqslant}

\title{Higher cluster tilting objects in locally finite triangulated categories}
\author{Qi Bin, Yifu Han, Wei Hu\thanks{Corresponding author: huwei@bnu.edu.cn}
}
\date{}

\begin{document}

\maketitle

\renewcommand{\thefootnote}{\alph{footnote}}
\setcounter{footnote}{-1} \footnote{2020 Mathematics Subject
Classification: 18E30;  16G70, 16G10.}
\renewcommand{\thefootnote}{\alph{footnote}}
\setcounter{footnote}{-1} \footnote{Keywords: Locally finite triangulated categories; $d$-cluster tilting objects;   derived equivalences.}

\begin{abstract}
We study higher cluster tilting objects through covering functors from derived
categories of hereditary algebras.  The covering formalism reduces the
existence problem to the equivariant problem of finding \(G\)-stable
\(d\)-cluster tilting objects in  \(d\)-cluster categories.  For
triangulated categories with finitely many indecomposable objects this gives a
complete ADE existence criterion and explicit counting formulas.  We also prove
that, in finite Frobenius models, the endomorphism algebras of all
\(d\)-cluster tilting objects are derived equivalent.  Applications are given
to Cohen--Macaulay finite categories, finite noncommutative crepant
resolutions, and rigidity dimensions of representation-finite self-injective
algebras.
\end{abstract}

\section{Introduction}

Highly rigid objects in triangulated or exact categories play an important
role in representation theory.  Typical examples are tilting modules in module categories and tilting
complexes in derived categories
\cite{Happel1988aa,Rickard1989ab,Keller2007aa}.  Higher cluster tilting
objects form another distinguished class of highly rigid objects: they are
higher analogues of tilting and cluster tilting objects in Iyama's higher
Auslander--Reiten theory \cite{Iyama2007ab}.  They provide a basic organizing
structure in higher representation-finite and preprojective algebras
\cite{Iyama2011ac,Iyama2013ab}, produce important examples and constructions in
Calabi--Yau and \(n\)-angulated categories
\cite{Keller2008ab,Geiss2013ab}, and appear in categorical models for cluster
algebras \cite{Buan2006aa,Buan2009ab,Buan2011aa}.  They are also closely
related to noncommutative crepant resolutions, where mutation and derived
equivalence of cluster tilting or maximal modification modules are central
tools \cite{VanDenBergh2004,Iyama2008ac,Iyama2013ac,Iyama2013aa}.

It is therefore natural, and already quite subtle, to ask whether a given
triangulated category or a module category of a finite-dimensional algebra
admits a \(d\)-cluster tilting object.  For self-injective algebras this
existence question is already highly restrictive: Erdmann and Holm proved that
if a self-injective algebra admits a maximal \(n\)-orthogonal module in the
sense of Iyama, then all \(A\)-modules have complexity at most \(1\)
\cite{ErdmannHolm2008}.  In this direction, Darp\"o and Iyama
studied \(d\)-cluster tilting modules over self-injective algebras via orbit
algebras of repetitive categories and covering theory; they also obtained
necessary and sufficient conditions for self-injective Nakayama algebras and
conjectured that, for a fixed finite-dimensional algebra, only finitely many
integers \(d\) can occur \cite{DarpoIyama2020}.  This is closely related to
another conjecture in \cite{ChenFangKernerKoenigYamagata2021}, which asserts
that every finite-dimensional algebra has finite rigidity dimension, a
homological invariant measuring the largest integer \(d\) for which the algebra
admits a \(d\)-rigid generator-cogenerator whose endomorphism algebra has
finite global dimension.  On the other hand, special
rigid objects often produce derived equivalences between their endomorphism
algebras.  This leads to a second natural question: whether the endomorphism
algebras of \(d\)-cluster tilting objects in a triangulated category, or in a
Frobenius model of it, must be derived equivalent.

In this paper, we consider a triangulated category \(\mathcal T\) equipped with a covering
from the derived category \(\Db[H]\) of a hereditary algebra, with fibres given
by the orbits of a group \(G\) of autoequivalences.  Under this covering, the
existence of \(d\)-cluster tilting objects in \(\mathcal T\) is transformed
into the existence of \(G\)-stable \(d\)-cluster tilting objects in the
\(d\)-cluster category \(\mathcal C_d(H)\).  When \(\mathcal T\) has only
finitely many indecomposable objects, the classification theorem gives an
Auslander--Reiten quiver of the form \(\mathbb Z\Delta/G\), where \(\Delta\) is
Dynkin and \(G\) is weakly admissible \cite{Xiao2005,Amiot2007}.  In this
finite situation the above reduction leads to a complete ADE answer.

The first main result is a complete ADE existence criterion.  Let \(L_G\) be
the order induced by a generator of \(G\), defined in Section 3 below.
\medskip

\noindent{\bf Theorem A} (Theorem \ref{theo-existence-dct}). {\em 
Let \(\mathcal T\) be a locally finite triangulated category whose
Auslander--Reiten quiver is \(\mathbb Z\Delta/G\), where \(\Delta\) is Dynkin
and \(G\) is weakly admissible.  Then \(\mathcal T\) admits a \(d\)-cluster
tilting object if and only if \(L_G\) satisfies the explicit ADE conditions in
Table \ref{tab:existence-conditions}.}

\medskip
The same analysis gives the labelled counting formulas in Theorem
\ref{theo-counting-dct}.  Thus the proof of sufficiency is constructive rather
than merely existential.  For \(d=1\), Corollary
\ref{coro-CY-generator-conditions} compares the criterion with the
finite-type \(2\)-Calabi--Yau classification of
Burban--Iyama--Keller--Reiten.

The second main result concerns endomorphism algebras.  Its point is that the
classification theorem is not only an existence result: it also gives enough
equivariant mutation connectedness to force derived equivalences in finite
Frobenius models.

\medskip
\noindent{\bf Theorem B} (Theorem \ref{theo-equivariant-mutation-connected}). {\em 
Let $k$ be an algebraically closed field, and let \(\mathcal T\) be a Hom-finite Krull--Schmidt triangulated $k$-category
with only finitely many indecomposable objects.  Assume that \(\mathcal T\)
admits \(d\)-cluster tilting objects and that
\(\mathcal T=\underline{\mathcal E}\) for a Frobenius category \(\mathcal E\)
with a projective generator.  Then the endomorphism algebras of any two
\(d\)-cluster tilting objects  in \(\mathcal E\) are derived equivalent.}

\medskip

The proof is organized around \(G\)-mutation.  First, Proposition
\ref{prop-G-mutation-derived-equivalence} shows that \(G\)-mutation reachable
\(G\)-stable \(d\)-cluster tilting objects give derived equivalent
endomorphism algebras after passing to the Frobenius model.  The remaining
point is to prove  
\(G\)-mutation connectedness for the relevant ADE \(d\)-cluster categories
(Proposition \ref{prop-dynkin-G-mutation-connected}).  Corollaries
\ref{coro-finite-frobenius-derived-equivalence} and
\ref{coro-CMfinite-isolated-NCCR} apply the theorem to Cohen--Macaulay finite
Iwanaga--Gorenstein algebras, representation-finite self-injective algebras,
and finite NCCR situations.

Finally, Section 5 applies the existence criterion to rigidity dimensions.
Combining Theorem A with the rigidity-degree formulas in
\cite{HuYinRigidityDegrees}, we obtain explicit type \(\mathbb A\) and type
\(\mathbb D\) families of representation-finite self-injective algebras for
which the maximal rigidity degree \(d_{\max}\) is realized by a
\(d_{\max}\)-cluster tilting module; consequently their rigidity dimensions are
\(d_{\max}+2\).  These families are given in
Corollaries \ref{coro-rigdim-type-A-closed-families} and
\ref{coro-rigdim-type-D-closed-families}.

\section{Preliminaries}
Throughout this paper, we fix an algebraically closed field $k$. 

\subsection{$d$-cluster tilting and covering}

{\bf $d$-cluster tilting}.  Let $\T$ be a triangulated category or an exact category and let $d\geq 1$ be an integer. If $\T$ is a triangulated category, we write $\Ext_{\T}^d(-,-)$ for $\Hom_{\T}(-,-[d])$. For a class $\mathcal{M}$ of objects in $\ind\T$, we write 
\[\mathcal{M}^{\perp_d}=\{X\in\ind\T\mid \Ext_{\T}^i(M,X)=0 \mbox{ for all }1\leq i\leq d, M\in\mathcal{M}\}.\]
One can similarly define ${}^{\perp_d}\mathcal{M}$. A class $\mathcal{M}$ of objects in $\ind\T$ is called a {\em $d$-cluster tilting subset} if $\add(\mathcal{M})$ is contravariantly finite in $\T$ and $\mathcal{M}=\mathcal{M}^{\perp_d}={}^{\perp_d}\mathcal{M}$. 
An object in $\T$ is called a {\em $d$-cluster tilting object} if the isomorphism classes of its indecomposable direct summands form a $d$-cluster tilting subset in $\ind\T$.  

Assume that $\T$ is triangulated and  has a Serre functor $\nu$. Then $\T$ has Auslander--Reiten triangles and the AR-translation $\tau$ is given by $\tau = \nu \circ [-1]$. For each $d\geq 1$, we have 
\(\Hom_{\T}(X,Y[d+1-i])\simeq D\Hom_{\T}(Y[d+1-i],\nu X)\simeq D\Hom_{\T}(Y,\tau X[i-d])\) for all $X,Y\in\T$. Set $\tau_{d+1}=\tau\circ[-d]$. Then the above isomorphism can be rewritten as \(\Hom_{\T}(X,Y[d+1-i])\simeq D\Hom_{\T}(Y,\tau_{d+1} X[i])\) for all $X,Y\in\T$. It follows that 
\(\tau_{d+1}\left({}^{\perp_d}\mathcal{M}\right)=\mathcal{M}^{\perp_d}.\)
for any $\mathcal{M}\subseteq\ind\T$. Thus a $d$-cluster tilting subset $\mathcal{M}$ in $\ind\T$ satisfies $\tau_{d+1}(\mathcal{M})=\mathcal{M}$. 

\medskip 
{\bf Covering functors}. We use covering functors in the sense introduced by
Bongartz and Gabriel \cite{BongartzGabriel1982}.  Let $\mathcal{A}$ and
$\mathcal{B}$ be Hom-finite $k$-linear categories, and let
$F:\mathcal{A}\lra \mathcal{B}$ be a $k$-linear functor which induces a
covering functor $F:\ind\mathcal{A}\lra\ind\mathcal{B}$, that is,  
  the canonical morphisms
\[\bigoplus_{Fy=Y}\Hom_{\mathcal{A}}(x,y)\lra \Hom_{\mathcal{B}}(Fx,Y),\quad \bigoplus_{Fx=X}\Hom_{\mathcal{A}}(x,y)\lra\Hom_{\mathcal{B}}(X,Fy)\]
induced by $F$ are bijections. The covering functor always gives rise to a push-down functor 
\(F_{\lambda}: \Modcat{\mathcal{A}}\lra\Modcat{\mathcal{B}}\) which is exact and left adjoint to the pull-up functor $F_*$ sending $M$ to $M\circ F$. Moreover, there is a natural isomorphism 
\(F_{\lambda}\Hom_{\mathcal{A}}(-,x)\lra \Hom_{\mathcal{B}}(-,Fx)\)
for all $x\in \mathcal{A}$.
\begin{Lem}\label{lem-covering-approx}
Keep the notations above. Let $f: x\lra y$ be a morphism in $\mathcal{A}$, and let $U\in\ind\mathcal{B}$. Then $\Hom_{\mathcal{A}}(u, f)$ (respectively, $\Hom_{\mathcal{A}}(f, u)$) is surjective for all $u\in F^{-1}(U)$ if and only if $\Hom_{\mathcal{B}}(U, Ff)$ (respectively, $\Hom_{\mathcal{B}}(Ff, U)$) is surjective.
\end{Lem}
\begin{proof}
We have the following commutative diagram.
\[\xymatrix@C=20mm{
\bigoplus\limits_{Fu=U}\Hom_{\cal A}(u,x) \ar[r]^{\bigoplus\Hom_{\cal A}(u,f)} \ar[d]  & \bigoplus\limits_{Fu=U}\Hom_{\cal A}(u,y) \ar[d] \\
\Hom_{\cal B}(U, Fx) \ar[r]^{\Hom_{\cal B}(U, Ff)} & \Hom_{\cal B}(U, Fy)
}\] 
Since $F:\ind\mathcal{A}\lra \ind\mathcal{B}$ is a covering functor, the vertical arrows are isomorphisms. Thus, the horizontal arrows are surjective if and only if the corresponding arrows in the bottom row are surjective. The proof for the case of $\Hom_{\mathcal{A}}(f, u)$ is similar.
\end{proof}
Now we assume   that $\mathcal{A}$ and $\mathcal{B}$ are triangulated categories, the covering functor $F$  commutes with $[1]$, that is, $F\circ [1]\simeq [1]\circ F$. It is not known whether $F$ is a triangle functor. Nevertheless, we have the following lemma. 
\begin{Lem}\label{lem-covering-triangle}
Let $F:\mathcal{A}\lra \mathcal{B}$ be a $k$-linear functor between two Hom-finite $k$-linear triangulated categories such that $F$ commutes with $[1]$ and that $F:\ind\mathcal{A}\lra \ind\mathcal{B}$ is a covering functor. Then for each triangle $x\lraf{f} y\lraf{g} z\lraf{h} x[1]$ in $\mathcal{A}$,  
there is a triangle $Fx\lraf{Ff} Fy\lra Fz\lra Fx[1]$ in $\mathcal{B}$.
\end{Lem}
\begin{proof}
Clearly we can assume that the triangle has no trivial triangles as direct summands, that is, both $f$ and $g$ are radical morphisms. The triangle induces a long exact sequence in $\Modcat{(\ind\mathcal{A})}$. Applying the push-down functor $F_{\lambda}$, which is exact, we get a long exact sequence in $\Modcat{(\ind\mathcal{B})}$. 
\[\xymatrix@C=13mm{
F_{\lambda}\Hom_{\cal A}(-,y)\ar[r]\ar[d] & F_{\lambda}\Hom_{\cal A}(-,z)\ar[r]\ar[d] & F_{\lambda}\Hom_{\cal A}(-,x[1])\ar[r]\ar[d] & F_{\lambda}\Hom_{\cal A}(-,y[1])\ar[d]\\
\Hom_{\cal B}(-,Fy)\ar[r]^{(-,Fg)} & \Hom_{\cal B}(-,Fz)\ar[r] & \Hom_{\cal B}(-,(Fx)[1])\ar[r]^{(-,(Ff)[1])} & \Hom_{\cal B}(-,(Fy)[1])\\
}\]
The vertical maps are induced by the natural isomorphisms $F_{\lambda}\Hom_{\cal A}(-,-)\simeq \Hom_{\cal B}(-,F(-))$ and $F\circ  [1]\simeq [1]\circ F$. 

Extend $Fx\lraf{Ff} Fy$ into a triangle $Fx\lraf{Ff} Fy\lraf{\eta} Z\lra (Fx)[1]$. It suffices to prove that $Z$ is isomorphic to $Fz$. The triangle then induces a long exact sequence in $\Modcat{(\ind{\cal B})}$.
\[\Hom_{\cal B}(-, Fx)\lraf{(-,Ff)}\Hom_{\cal B}(-,Fy)\lraf{(-,\eta)}\Hom_{\cal B}(-, Z)\lra \Hom_{\cal B}(-,(Fx)[1])\lraf{(-,(Ff)[1])}\cdots\] 
Combine the above two long exact sequences, both $\Hom_{\cal B}(-, Fz)$ and $\Hom_{\cal B}(-, Z)$ are projective covers of $\Ker (-,(Ff)[1])$ so that they are isomorphic. The Yoneda lemma then tells us that $Fz\simeq Z$. This finishes the proof. 
\end{proof}

The following lemma describes $d$-cluster tilting subsets under covering functors. 

\begin{Lem}\label{lem-covering-dct}
Suppose that $\mathcal{C}$ and $\T$ are Hom-finite $k$-linear triangulated categories and that $F:\mathcal{C}\lra\T$ is a $k$-linear functor such that $F\circ [1]\simeq [1]\circ F$ and that $F$ induces a covering from $\ind\mathcal{C}$ to $\ind\T$.  Then for each $d\geq 1$, a subset $\mathcal{M}\subseteq\ind\T$ is a $d$-cluster tilting subset in $\ind\T$ if and only if $F^{-1}(\mathcal{M})$ is a $d$-cluster tilting subset in $\ind\mathcal{C}$.
\end{Lem}
\begin{proof}
For each $x,y\in \ind\mathcal{C}$, since $F$ is a covering functor, we have 
\[\Hom_{\T}(Fx, (Fy)[i])\simeq\Hom_{\T}(Fx,F(y[i]))=\bigoplus_{Fz=Fy}\Hom_{\mathcal{C}}(x,z[i]).\]
Thus $Fy\in (Fx)^{\perp_d}$ if and only if $z\in x^{\perp_d}$ for all $z\in\ind\mathcal{C}$ with $Fz=Fy$. Similarly, one has $Fx\in {}^{\perp_d}Fy$ if and only if $z\in {}^{\perp_d}y$ for all $z\in\ind\mathcal{C}$ with $Fz=Fx$. Thus, for $X,Y\in\ind\T$, we have $\Hom_{\T}(X,Y[i])=0$ for all $1\leq i\leq d$ if and only if $\Hom_{\mathcal{C}}(x,y[i])=0$ for all $1\leq i\leq d$ and for all $x\in F^{-1}(X)$ and $y\in F^{-1}(Y)$. Using this fact, the lemma follows easily. 
\end{proof}

\subsection{Locally finite triangulated categories}
\label{subsection-locally-finite}

Given a Dynkin diagram $\Delta$ of type $ADE$, we label the vertices of $\Delta$ and fix an orientation as follows:
$$\xymatrix@=1.5em{
& \mathbb{A}_n: \quad 1 \ar[r] & 2 \ar[r]&  \cdots \ar[r] &  \ar[r] n-1 &  n & n\geq 1
}$$
$$\xymatrix@=1.5em{
& & & &  & n-1\\
& \mathbb{D}_n: \quad 1 \ar[r] & 2 \ar[r]&  \cdots \ar[r] &  n-2\ar[ru]\ar[rd]  & & n\geq 4\\
& & & & &  n
}$$
$$\xymatrix@=1.5em{
& & &   4 \\
&  \mathbb{E}_n: \quad  1 & 2  \ar[l] & 3 \ar[l]\ar[r]\ar[u] &  5 \ar[r] &  \cdots \ar[r] & n & n=6,7,8
}$$
One can construct a translation quiver $\mathbb{Z}\Delta$ as follows: the vertices of $\mathbb{Z}\Delta$ are pairs $(i, p)$ where $i\in \mathbb{Z}$ and $p$ is a vertex of $\Delta$. There is an arrow from $(i, p)$ to $(i, q)$ if there is an arrow from $p$ to $q$ in $\Delta$, and there is an arrow from $(i, p)$ to $(i + 1, q)$ if  there is an arrow from $q$ to $p$. The translation $\tau$ on $\mathbb{Z}\Delta$ is defined by $\tau(i, p) = (i - 1, p)$.

The mesh category $k(\mathbb{Z}\Delta)$ of $\mathbb{Z}\Delta$ is defined as follows: the objects of $k(\mathbb{Z}\Delta)$ are the vertices of $\mathbb{Z}\Delta$, and the morphisms are generated by the arrows of $\mathbb{Z}\Delta$ subject to the mesh relations. It is well-known that $k(\mathbb{Z}\Delta)$ is equivalent to $\ind \Db[k\Delta]$, where $\Db[k\Delta]$ is the bounded derived category of the path algebra $k\Delta$ and $\ind\Db[k\Delta]$ is the full subcategory of $\Db[k\Delta]$ consisting of one representative from each isomorphism class of indecomposable objects (see, for example, \cite{Happel1988aa}). In particular, the AR-quiver of $\Db[k\Delta]$ is isomorphic to $\mathbb{Z}\Delta$. The shift functor $[1]$ on $\Db[k\Delta]$ induces an automorphism $S$ of $\mathbb{Z}\Delta$. We list the automorphism $S$ for each type of $\Delta$ as follows(see, for example, \cite[2.2]{Amiot2007}):
\begin{itemize}
	\item[--] if $\Delta = \mathbb{A}_n$, then $S(i, p) = (i + p, n + 1 - p)$; 
	\item[--] if $\Delta = \mathbb{D}_n$ with $n$ even, then $S = \tau^{-n+1}$; 
	\item[--] if $\Delta = \mathbb{D}_n$ with $n$ odd, then $S = \tau^{-n+1}\phi$ where $\phi$ is the automorphism of $\mathbb{D}_n$ which exchanges $n$ and $n - 1$; 
	\item[--] if $\Delta = \mathbb{E}_6$, then $S = \phi\tau^{-6}$ where $\phi$ is the automorphism of $\mathbb{E}_6$ which exchanges $2$ and $5$, and $1$ and $6$; 
	\item[--] if $\Delta = \mathbb{E}_7$, then $S = \tau^{-9}$; 
	\item[--] if $\Delta = \mathbb{E}_8$, then $S = \tau^{-15}$.
\end{itemize}
For each vertex $x\in\mathbb{Z}\Delta$, we denote by $x^+$ the set of vertices $y$ such that there is an arrow from $x$ to $y$, and we denote by $x^-$ the set of vertices $y$ such that there is an arrow from $y$ to $x$.  An automorphism $G$ of $\mathbb{Z}\Delta$ is said to be {\em weakly admissible} \cite{Dieterich1987} if, for each $g \in G \backslash \{1\}$ and for each $x \in \mathbb{Z}\Delta$, we have $x^+ \cap (gx)^+ = \emptyset$.

\medskip 
{\bf Locally finite triangulated categories}. Let $\T$ be a Krull--Schmidt  triangulated $k$-category. $\T$ is called {\em locally finite} if for each indecomposable object $X$ of $\T$, there are only finitely many isomorphism classes of indecomposable objects $Y$ such that $\Hom_{\T}(X, Y) \neq 0$. Note that if $\T$ is a locally finite triangulated category, then $\T$ has Auslander--Reiten triangles \cite{Xiao2002},
and the AR-quiver of $\T$ is a disjoint union of connected components, each of which is of the form $\mathbb{Z}\Delta/G$, where $\Delta$ is a Dynkin tree and $G$ is a weakly admissible group of automorphisms of $\mathbb{Z}\Delta$ (see \cite{Xiao2005,Amiot2007}). All possible types of the AR-quiver of $\T$ are listed in the following theorem. 
\begin{Theo}[{\cite[Theorem 2.1]{Amiot2007}}]
\label{thm-locally-finite-classification}
Let $\Delta$ be a Dynkin tree and $G$ a non-trivial group of weakly admissible  automorphisms of $\mathbb{Z}\Delta$. Then $G$ is cyclic of infinite order and here is a list of its possible generators:
\begin{itemize}
    \item if $\Delta=\mathbb{A}_n$ with n odd, possible generators are $\tau^{r}$ and $\phi \tau^{r}$ with $r\geq 1$, where $\phi=\tau^{\frac{1}{2}(n+1)}S$ is an automorphism of $\mathbb{Z}\Delta$ of order $2$ ;
    \item if $\Delta=\mathbb{A}_n$ with n even, then possible generators are $\rho^r$, where $r\geq 1$ and where $\rho=\tau^{\frac{1}{2}n}S$;
    \item if $\Delta=\mathbb{D}_n$ with $n\geq 5$, then possible generators are $\tau^{r}$ and $ \tau^{r}\phi$ where $r\geq 1$ and where $\phi=(n-1,n)$ is the automorphism of $\mathbb{D}_n$ exchanging $n$ and $n-1$;
    \item if \(\Delta = \mathbb{D}_4\), then possible generators are \(\phi \tau^r\), where \(r \geq 1\) and where \(\phi\) belongs to \(\mathcal{S}_3\) the permutation group on three elements seen as subgroup of automorphisms of \(\mathbb{D}_4\);
\item if \(\Delta = \mathbb{E}_6\), then possible generators are \(\tau^r\) and \(\phi \tau^r\), where \(r \geq 1\) and where \(\phi\) is the automorphism of \(\mathbb{E}_6\) exchanging $2$ and $5$, and $1$ and $6$;
\item if \(\Delta = \mathbb{E}_n\) with \(n = 7, 8\), possible generators are \(\tau^r\), where \(r \geq 1\).
\end{itemize}
\end{Theo}

\subsection{$d$-cluster categories and geometric models}
\label{subsec-geometric-model}

Given a hereditary algebra $H$, the $d$-cluster category $\mathcal{C}_d(H)$ associated to $H$ is the orbit category $\Db[H]/\langle\tau_{d+1}\rangle$, where $\tau_{d+1}=\tau\circ[-d]$. $\mathcal{C}_d(H)$ is triangulated and the canonical functor \(\pi:\Db[H]\lra\mathcal{C}_d(H)\) is a triangle functor (see \cite{Keller2005ab}).  Baur and Marsh established geometric models for $d$-cluster categories of  types  $\mathbb{A}$ and $\mathbb{D}$ in  \cite{Baur2007,Baur2008}.

For $H=k\mathbb{A}_{n}$,  set \(N=(n+1)d+2\) and consider the regular $N$-gon $P_N$ whose vertices are labeled clockwise by $0, 1, \ldots, N-1$. Then the indecomposable objects in $\mathcal{C}_d(k\mathbb{A}_n)$ correspond bijectively to those $d$-diagonals of $P_N$, that is, those edges $D_{(x,y)}$ connecting $x$ and $y$ with $y-x>1$ and $y-x-1\in d\mathbb{Z}$. Here we identify $x$ with $x'$ if $x-x'\in N\mathbb{Z}$. Suppose $X\in\mathcal{C}_d(k\mathbb{A}_n)$ is the object corresponding to the $d$-diagonal $D_{(x,y)}$. Then $\tau X$ (respectively, $X[1]$) corresponds to the $d$-diagonal $D_{(x-d,y-d)}$ (respectively, $D_{(x-1,y-1)}$) obtained by rotating $D_{(x,y)}$ around the center by $\frac{2d\pi}{N}$ (respectively, $\frac{2\pi}{N}$) anti-clockwise.

A maximal collection of pairwise non-crossing  $d$-diagonals of $P_{N}$ is called a  $(d+2)$-angulation of $P_{N}$.  Basic $d$-cluster tilting objects in $\mathcal{C}_d(k\mathbb{A}_n)$ correspond bijectively to $(d+2)$-angulations of $P_N$ (see \cite{Thomas2007,Zhu2008}).

The geometric model for type \(\mathbb D\) is as follows.  Put
\(N=(n-1)d+1\)
and let \(P_N^{\circ}\) be a regular \(N\)-gon with one puncture at the centre.
The boundary vertices are labelled by \(0,1,\ldots,N-1\) clockwise. Each integer $i$ corresponds to a boundary vertex $[i]_N$, where $[i]_N$ is the remainder of $i$ modulo $N$. For $i\neq j$, let \(B_{ij}\) denote the clockwise boundary path from \(i\) to \(j\).

\begin{Def}
A \(d\)-arc in \(P_N^{\circ}\) is one of the following. 
\begin{enumerate}
    \item For $i<j<i+N$ with $j-i-1\in d\mathbb{Z}$, the $d$-arc \(D_{ij}\) is the homotopy class, relative to the
endpoints, of the curve from \(i\) to \(j\) homotopic to \(B_{ij}\), not
meeting the boundary except at its endpoints and not passing through the
puncture.
\item For each boundary vertex \(i\), there are two tagged radii
\(D_{ii}^{\epsilon}\) joining \(i\) to the puncture, where $\epsilon\in\{-1,1\}$.
\end{enumerate}
\end{Def}
Here the terminology for tagged $d$-arcs is slightly different from that in Baur and Marsh's original paper. For simplicity, we write tag $1$ and tag $-1$ as $+$ and $-$, respectively; thus \(D_{ii}^{+}=D_{ii}^{1}\) and \(D_{ii}^{-}=D_{ii}^{-1}\).
Baur and Marsh prove that indecomposable objects of the $d$-cluster category $\mathcal{C}_d(k\mathbb{D}_n)$ correspond bijectively to the $d$-arcs. The tagged radii $D_{ii}^{\pm}$ correspond to high vertices on the AR-quiver. Under this bijection, the arrows on the AR-quiver are described as $d$-moves of $d$-arcs. There are $4$ types of $d$-moves, as listed below. 
\[D_{i,j}\rightarrow D_{i,j+d},\quad D_{i,j}\rightarrow D_{i+d,j},\quad D_{i,i+N-d}\rightarrow D_{ii}^{\pm} \mbox{ and }D_{ii}^{\pm}\rightarrow D_{i+d,i+N}\]
The AR-translation is obtained by shifting indices by $-d$, that is, rotating anti-clockwise by $d$ vertices, and changing the tag if $d$ is odd. The shift functor $[1]$ rotates a $d$-arc anti-clockwise by $1$ step, and changes the tag unless both $n$ and $d$ are even. For a $d$-arc $\alpha$, we write $X_{\alpha}$ for the corresponding object in $\mathcal{C}_d(k\mathbb{D}_n)$. The Hom-hammock starting at a $d$-arc $\alpha$ contains precisely those $d$-arcs $\beta$ with $\Hom_{\mathcal{C}_d(k\mathbb{D}_n)}(X_{\alpha},X_{\beta})\neq 0$. The Hom-hammock has the following shape. 
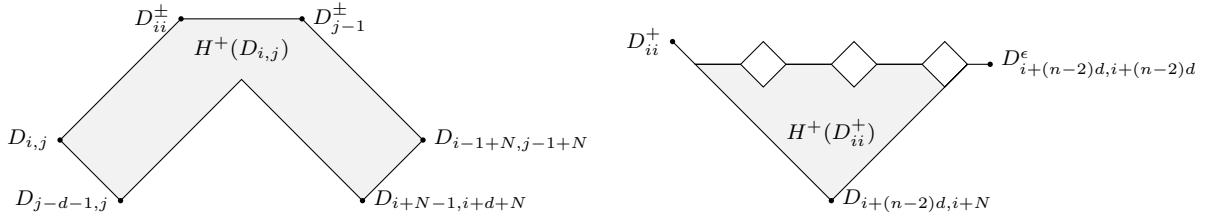
\begin{center}
	\begin{tikzpicture}
\begin{scope}
	\pgfmathsetmacro{\ul}{0.2};
	\draw [-,fill=gray!10] (-8*\ul,0)--(-12*\ul,4*\ul)--(-4*\ul,12*\ul)--(4*\ul,12*\ul)--(12*\ul,4*\ul)--(8*\ul,0*\ul)--(0*\ul,8*\ul)-- cycle;
	\fill[fill=black]  (-8*\ul,0) circle (1pt) node[left] {$\scriptstyle{D_{j-d-1,j}}$};
	\fill[fill=black]  (-12*\ul,4 *\ul) circle (1pt) node[left] {$\scriptstyle{D_{i,j}}$};
	\fill[fill=black]  (-4*\ul,12*\ul) circle (1pt) node[left] {$\scriptstyle{D_{ii}^{\pm}}$};
	\fill[fill=black]  (12*\ul,4*\ul) circle (1pt) node[right] {$\scriptstyle{D_{i-1+N,j-1+N}}$};
	\fill[fill=black]  (4*\ul,12*\ul) circle (1pt) node[right] {$\scriptstyle{D_{j-1}^{\pm}}$};
	\fill[fill=black]  (8*\ul,0*\ul) circle (1pt) node[right] {$\scriptstyle{D_{i+N-1,i+d+N}}$};
	\node at (0*\ul,10*\ul) {$\scriptstyle{ H^+(D_{i,j})}$};
\end{scope}
\begin{scope}[shift={(6,0)}]
    \pgfmathsetmacro{\ul}{0.3};
\draw [-,fill=gray!10] (0*\ul,6*\ul)--(2*\ul,6*\ul)--(3*\ul,5*\ul)--(4*\ul,6*\ul)--(6*\ul,6*\ul)--(7*\ul,5*\ul)--(8*\ul,6*\ul)--(10*\ul,6*\ul)--(11*\ul,5*\ul)--(12*\ul,6*\ul)--(6*\ul,0*\ul)--cycle;
\draw (-1*\ul,7*\ul)--(0*\ul,6*\ul)--(2*\ul,6*\ul)--(3*\ul,7*\ul)--(4*\ul,6*\ul)--(6*\ul,6*\ul)--(7*\ul,7*\ul)--(8*\ul,6*\ul)--(10*\ul,6*\ul)--(11*\ul,7*\ul)--(12*\ul,6*\ul)--(13*\ul,6*\ul);
\fill[fill=black]  (13*\ul,6*\ul) circle (1pt) node[right] {${\scriptstyle D_{i+(n-2)d,i+(n-2)d}^{\epsilon}}$};
\fill[fill=black]  (11*\ul,7*\ul);
\fill[fill=black]  (-1*\ul,7*\ul) circle (1pt) node[left] {${\scriptstyle D_{ii}^{+}}$};
\fill[fill=black]  (6*\ul,0*\ul) circle (1pt) node[right] {${\scriptstyle D_{i+(n-2)d,i+N}}$};
\node at (6*\ul,3*\ul) {${\scriptstyle H^+(D_{ii}^{+})}$}; 
\end{scope}
	\end{tikzpicture}
	\end{center}
The tag $\epsilon$ is   \(+\) when $n$ is even, and is   \(-\) when $n$ is odd.  Two $d$-arcs $\alpha$ and $\beta$  are said to be $d$-orthogonal if $X_{\alpha}\in{}^{\perp_d}X_{\beta}$, or equivalently $X_{\beta}\in{}^{\perp_d}X_{\alpha}$.
\begin{Lem}\label{lem-non-crossing}
Let $\alpha$ and $\beta$ be two $d$-arcs. Then the following hold. 
\begin{enumerate}
    \item If at least one of $\alpha$ and $\beta$ is non-tagged, then $\alpha$ and $\beta$ cross in $P_N^\circ$ if and only if, after possibly interchanging $\alpha$ and $\beta$, there is some integer $0<t\leq d$ such that $\beta[t]$ belongs to $H^+(\alpha)$, that is, $\alpha$ and $\beta$ are not $d$-orthogonal. 
    \item If \(\alpha=D_{ii}^{\epsilon}\) is tagged and \(\beta=\tau^{-t}(D_{ii}^{\epsilon'})\), where \(\epsilon,\epsilon'\in\{1,-1\}\), for some integer \(0<t<N\), then \(\alpha\) and \(\beta\) are \(d\)-orthogonal 
    if and only if 
	    \(\epsilon\epsilon'=(-1)^{t+\left\lceil\frac{t}{n-1}\right\rceil+1}.\)
   \item For $i<j<i+N$, the object  $D_{ii}^{\pm}$ is $d$-orthogonal to precisely one of $D_{jj}^{+}$ and $D_{jj}^{-}$. 
\end{enumerate}
\end{Lem}
\begin{proof}
Assume that $\alpha=D_{i,j}$ is an ordinary $d$-arc with $i<j<i+N$. If $\beta$ is an ordinary $d$-arc such that $\alpha$ and $\beta$ cross in $P_N^\circ$, then, without loss of generality, we may assume that $\beta=D_{ml}$ with $i<m<j<l<m+N$. Let $t\in\{1,2,\cdots,d\}$ be the unique number such that $m-t-i\in d\mathbb{Z}$. It follows that $l-t-j\in d\mathbb{Z}$ since both $j-i-1$ and $l-m-1$ are divided by $d$.  Then $\beta[t]=D_{m-t,l-t}$ and there is a path in the Hom-hammock $H^+(\alpha)$ from $\alpha$ to $\beta[t]$. The path reads as follows. 
\[D_{i,j}\rightarrow D_{i+d,j}\rightarrow\cdots\rightarrow D_{m-t,j}\rightarrow D_{m-t,j+d}\rightarrow\cdots\rightarrow D_{m-t,l-t}.\]
It follows that $\beta[t]\in H^+(\alpha)$ and thus $\Hom_{\mathcal{C}_d(k\mathbb{D}_n)}(X_{\alpha}, X_{\beta}[t])\neq 0$.  Conversely, if $\beta[t]\in H^+(\alpha)$ for some $0<t\leq d$, then one can directly check that $\alpha$ and $\beta$ cross each other in $P_N^{\circ}$ by the presentation of $H^+(D_{i,j})$ above. When $\beta$ is a tagged $d$-arc, we can prove (1) similarly. 

(2) follows straightforwardly by verifying the Hom-hammocks $H^-(\alpha[t])=H^+(\tau^{-1}\alpha[t-1])$ for $t=1,2,\cdots,d$. (3) follows from (2). 
\end{proof}

A \((d+2)\)-angulation of \(P_N^{\circ}\) is a maximal collection of
pairwise \(d\)-orthogonal \(d\)-arcs in the above sense; these collections are
in bijection with basic \(d\)-cluster tilting objects in
\(\mathcal C_d(k\mathbb D_n)\). By a result of Thomas
\cite[Theorem 1]{Thomas2007}, maximal \(d\)-rigid objects coincide with
\(d\)-cluster tilting objects. We refer to
\cite{Schiffler2008,Baur2007,Baur2008,Zhou2009,Thomas2007} for more details of
the geometric models and their relation with \(d\)-cluster combinatorics.

\section{Existence of $d$-cluster tilting objects}

In this section, we assume that $\T$ is a Hom-finite Krull--Schmidt triangulated $k$-category with a Serre functor, together with a covering functor \(F:\ind\Db[H]\lra \ind\T\)
for some finite dimensional hereditary $k$-algebra $H$ such that $F\circ[1]\simeq [1]\circ F$. We assume that there is a group $G\leq\Aut\Db[H]$ such that, for each $x\in\ind\Db[H]$, we have
\(Gx=F^{-1}(Fx)\).
  Our main question is: When does $\T$ admit $d$-cluster tilting objects?

\medskip 
Note that any element $\sigma\in G$ necessarily commutes with $\tau$ and $[1]$ on $\Db[H]$, it follows that the action of $G$ on $\Db[H]$ induces an action on $\mathcal{C}_d(H)=\Db[H]/\langle\tau_{d+1}\rangle$. 
\[\xymatrix{
&\Db[H]\ar[ld]_{F}\ar[rd]^{\pi}\\
\T &&\mathcal{C}_d
(H)}\]
We use \(G\)-stable to mean preserved by the action of \(G\) up to
isomorphism.  Thus a basic object is \(G\)-stable precisely when the set of
isomorphism classes of its indecomposable direct summands is \(G\)-stable.
\begin{Lem}
\label{lem-dct-H-ctcat}
Keep the notations above. The assignments
\[
\mathcal U\longmapsto \pi F^{-1}(\mathcal U)
\qquad\text{and}\qquad
\mathcal V\longmapsto F\pi^{-1}(\mathcal V)
\]
give mutually inverse bijections between the $d$-cluster tilting subsets of
\(\ind\mathcal T\) and the $G$-stable $d$-cluster tilting subsets of
\(\ind\mathcal C_d(H)\).  Moreover, suppose that
\(\tau_{d+1}^m\in G\) for some integer \(m>0\).  Then these assignments
induce mutually inverse bijections between
the isomorphism classes of basic $d$-cluster tilting objects in
\(\mathcal T\) and those of $G$-stable basic $d$-cluster tilting objects in
\(\mathcal C_d(H)\).
\end{Lem} 
\begin{proof}
By Lemma \ref{lem-covering-dct}, a subset
\(\mathcal U\subseteq\ind\mathcal T\) is $d$-cluster tilting if and only if
\(F^{-1}(\mathcal U)\) is $d$-cluster tilting in \(\ind\Db[H]\).  Every
$d$-cluster tilting subset of \(\Db[H]\) is
\(\tau_{d+1}\)-stable.  Therefore
\(F^{-1}(\mathcal U)\) is a union of fibres of \(\pi\), and another
application of Lemma \ref{lem-covering-dct} shows that
\(\pi F^{-1}(\mathcal U)\) is $d$-cluster tilting in
\(\ind\mathcal C_d(H)\).  It is $G$-stable because
\(F^{-1}(\mathcal U)\) is $G$-stable.

Conversely, let \(\mathcal V\subseteq\ind\mathcal C_d(H)\) be a $G$-stable
$d$-cluster tilting subset.  Applying Lemma \ref{lem-covering-dct} first to
\(\pi\) and then to \(F\), we see that
\(F\pi^{-1}(\mathcal V)\) is $d$-cluster tilting in
\(\ind\mathcal T\); here the $G$-stability of \(\mathcal V\) gives
\(F^{-1}F\pi^{-1}(\mathcal V)=\pi^{-1}(\mathcal V)\).  Finally,
\[
F\pi^{-1}\pi F^{-1}(\mathcal U)=\mathcal U,
\qquad
\pi F^{-1}F\pi^{-1}(\mathcal V)=\mathcal V,
\]
so the two assignments are mutually inverse.

For the last assertion, let \(U\) be a basic $d$-cluster tilting object in
\(\mathcal T\).  By the first part of the lemma,
\(\pi F^{-1}(\add(U)\cap\ind\mathcal T)\) is a $d$-cluster tilting subset of
\(\ind\mathcal C_d(H)\).  Every such subset has exactly as many
indecomposable objects as there are isomorphism classes of simple
$H$-modules.  It is therefore finite and is the set of
indecomposable summands of a $G$-stable basic $d$-cluster tilting object in
\(\mathcal C_d(H)\).  Conversely, the inverse image under \(\pi\) of the
finite set of indecomposable summands of a $G$-stable basic $d$-cluster
tilting object is a finite union of
\(K=\langle\tau_{d+1}\rangle\)-orbits and is $G$-stable.  Since
\(\langle\tau_{d+1}^m\rangle\subseteq G\), each $K$-orbit meets at most
$m$ $G$-orbits.  Thus its image under \(F\) is the set of indecomposable
summands of a basic object in
\(\mathcal T\).  The first part of the lemma shows that these objects are
$d$-cluster tilting and that the two constructions are mutually inverse.
\end{proof}

Note that all $d$-cluster tilting objects of $\mathcal{C}_d(H)$ are $\tau_{d+1}$-invariant. An immediate consequence is the following corollary. 
\begin{Coro}
\label{koro-taud-power}
Let $H$ be a finite dimensional hereditary algebra and let $d\geq 1$ be an integer. If $G=\langle\tau_{d+1}^m\rangle$ for some positive integer $m$, then $\T=\Db[H]/G$ admits $d$-cluster tilting objects. 
\end{Coro}
\begin{proof}
The canonical basic $d$-cluster tilting object of \(\mathcal C_d(H)\) is
$G$-stable, since $G=\langle\tau_{d+1}^m\rangle$ acts trivially on
\(\mathcal C_d(H)\).  Since \(\tau_{d+1}^m\in G\), the last assertion of
Lemma \ref{lem-dct-H-ctcat} gives a basic $d$-cluster tilting object in
\(\mathcal T\).
\end{proof}
 
Since every locally finite triangulated category admits a covering functor from the derived category of a hereditary algebra of Dynkin type, it is natural to ask the following question. 

\medskip 
\noindent
{\bf Question.} {\em Let $\T$ be a locally finite triangulated category, for which $d\geq 1$ does $\T$ admit $d$-cluster tilting objects?}

\medskip 
By the classification of locally finite triangulated categories, the AR-quiver of $\T$ is  of the form $\mathbb{Z}\Delta/G$ (see \cite{Xiao2005,Amiot2007}), where $\Delta$ is a Dynkin tree and $G=\langle\eta\rangle$ is a cyclic group listed in Theorem \ref{thm-locally-finite-classification}.  Such a weakly admissible group satisfies the additional condition in Lemma \ref{lem-dct-H-ctcat}.  Indeed, the list in Theorem \ref{thm-locally-finite-classification} shows that some positive power of \(\eta\) is a non-zero power of \(\tau\).  The formulas for the suspension in \S\ref{subsection-locally-finite} likewise show that some positive power of \(\tau_{d+1}=\tau S^{-d}\) is a non-zero power of \(\tau\).  These two pure translation subgroups have a common non-zero power.  Hence \(\tau_{d+1}^m\in G\) for some \(m>0\).

We list some conventions for the rest of this section. Let $h_{\Delta}$ be the Coxeter number of $\Delta$, as listed in the following table. 
\begin{center}
\begin{tabular}{c|c|c|c|c|c}
\toprule  
$\Delta$ & $\mathbb{A}_n$ & $\mathbb{D}_n$ & $\mathbb{E}_6$ & $\mathbb{E}_7$ & $\mathbb{E}_8$ \\
\hline
$h_{\Delta}$ & $n+1$ & $2n-2$ & $12$ & $18$ & $30$\\
\bottomrule
\end{tabular}
\vspace{0.5em}
\captionof{table}{Coxeter numbers.}
\label{tab:coxeter-numbers}
\end{center} 
We define 
\[N_d(\Delta)=\begin{cases}
    h_{\Delta}d+2, & \text{type }\mathbb{A};\\
    h_{\Delta}^*d+1, & \text{type }\mathbb{D},\mathbb{E},
\end{cases}\tag{$3.1$}\label{def-N}\]
where $h_{\Delta}^*=h_{\Delta}/2$ is the half Coxeter number. For  $\Delta$ of types $\mathbb{A}$ and $\mathbb{D}$, the number $N_d(\Delta)$ coincides with the $N$ in the geometric models (see \S\ref{subsec-geometric-model}).

\begin{Lem}\label{lem-N}
Let $\Delta$ be a Dynkin tree of type $\mathbb{D}$ or $\mathbb{E}$ listed in \S \ref{subsection-locally-finite}, and let $d\geq 1$ be an integer. Suppose that  $\phi$ is the identity functor for $\Delta$ not of type $\mathbb{D}_n$ with $n$ odd or $\mathbb{E}_6$, and $\phi$ is the automorphism of order $2$ for $\Delta$ of type $\mathbb{D}_n$ with $n$ odd or $\mathbb{E}_6$. Then we have $\tau_{d+1}=\tau^{N_d(\Delta)}\phi^d$ in $\Db[k\Delta]$, and 
$[N_d(\Delta)]\phi=\id$  in $\mathcal{C}_d(k\Delta)$.
\end{Lem}
\begin{proof}
The lemma follows from a direct check of the relations between $S$ and $\tau$
in \S \ref{subsection-locally-finite}.  For instance, in type
\(\mathbb D_n\) with \(n\) odd, we have
\(S=\tau^{-(n-1)}\phi\), where \(\phi\) is the automorphism of order \(2\).
Hence
\[
\tau_{d+1}=\tau S^{-d}=\tau^{(n-1)d+1}\phi^d
=\tau^{N_d(\mathbb D_n)}\phi^d .
\]
Moreover \(N_d(\mathbb D_n)\) is odd, and the relation
\(\tau^{N_d(\mathbb D_n)}=\phi^d\) holds in the cluster category
\(\mathcal C_d(k\mathbb D_n)\).  Therefore
\[
[N_d(\mathbb D_n)]\phi
=S^{N_d(\mathbb D_n)}\phi
=\tau^{-(n-1)N_d(\mathbb D_n)}\phi^{N_d(\mathbb D_n)+1}
=(\phi^d)^{-(n-1)}=\id .
\]
The cases \(\mathbb E_6\) and \(\phi=\id\) are checked in the same way.
\end{proof}

Let $\Aut(\mathbb{Z}\Delta)$ be the group consisting of weakly admissible automorphisms of $\mathbb{Z}\Delta$. We define a grading on $\Aut(\mathbb{Z}\Delta)$ as follows.  
\[\begin{array}{rcl}
    \deg:\Aut(\mathbb{Z}\Delta) & \lra & \mathbb{Z}/N_d(\Delta)\mathbb{Z}\\
    \tau & \mapsto & d,\\
    S & \mapsto & 1,\\
   \psi & \mapsto & 0,
\end{array}
\]
where  $\psi$ is induced by an automorphism of $\mathbb{D}_n$ or $\mathbb{E}_6$. It is easy to check that $\deg$ is well-defined by the relations in $\Aut(\mathbb{Z}\Delta)$ (see \S \ref{subsection-locally-finite}). If $G=\langle\eta\rangle\leq \Aut(\mathbb{Z}\Delta)$ is a weakly admissible group, then we define $L_G$ to be the order of $\deg(\eta)$ in $\mathbb{Z}/N_d(\Delta)\mathbb{Z}$, which is 
\[L_G=\frac{N_d(\Delta)}{(\deg(\eta),N_d(\Delta))}.\]
This is the smallest positive integer such that $\deg(\eta) L_G \equiv 0 \pmod{N_d(\Delta)}$. Note that $L_G$ is well-defined and depends only on $G$.

\medskip 
The   existence  of $d$-cluster tilting objects in a locally finite triangulated category $\T$ is classified by the following theorem. 
\begin{Theo}\label{theo-existence-dct}
Let $\T$ be a locally finite triangulated category, and let the AR-quiver of $\T$ be isomorphic to $\mathbb{Z}\Delta/G$, where $\Delta$ is a Dynkin tree and $G=\langle\eta\rangle$ is a weakly admissible group. Then $\T$ has a $d$-cluster tilting object if and only if $L_G$ satisfies the conditions in Table \ref{tab:existence-conditions}.
\end{Theo}

\begin{center}
\centering
\renewcommand{\arraystretch}{1.25}
\begin{tabular}{c|c|l||p{0.40\textwidth}}
\toprule
$\Delta$ & $\eta$ & parity of $(n,d)$ &condition\\
\hline
$\mathbb{A}_n$ & & & $L_G \mid (n,d+2)$, or $L_G=2$ and $n$ odd\\
\hline 
	\multirow{3}{*}{$\mathbb{D}_n$} & \multirow{3}{*}{$\tau^r$} & $d$ is even & $L_G\mid n$\\
	&& even, odd & $L_G\mid n$ and $n/L_G$ is even\\
	&& odd, odd & $L_G=1$, or $L_G\mid n$ and $rL_G/N$ is even\\
\hline
\multirow{3}{*}{$\mathbb{D}_n$} & \multirow{3}{*}{$\tau^r\psi,\, o(\psi)=2$} & $d$ is even & $L_G=1$\\
&& even, odd &  $L_G=1$, or $L_G\mid n$ and $n/L_G$ is odd\\
&& odd, odd & $L_G=1$, or $L_G\mid n$ and $rL_G/N$ is odd\\
\hline
$\mathbb{D}_4$ & $\tau^r\psi,o(\psi)=3$ & & $L_G=1$\\
\hline
\(\mathbb{E}_6\) & & & \(L_G=1\)\\
\hline
\(\mathbb{E}_7\) & & & \(L_G\in\{1,7\}\)\\
\hline
$\mathbb{E}_8$ & & & $L_G\in\{1,2,4\}$\\
\hline 
\multicolumn{4}{c}{\parbox{0.86\textwidth}{\centering
\(\psi\) denotes the automorphism of \(\mathbb Z\Delta\) induced by an
automorphism of \(\Delta\).}}\\
\bottomrule
\end{tabular}
\vspace{0.5em}
\captionof{table}{Existence conditions for \(d\)-cluster tilting objects.}
\label{tab:existence-conditions}
\end{center}

\subsection{Necessity}\label{subsection-necessarity}
We first prove the necessary part of Theorem \ref{theo-existence-dct}.  The
sufficiency will follow from the counting theorem in \S\ref{subsection-counting} below.  For simplicity, we shall write $N$ for $N_d(\Delta)$ from now on. 

\begin{Lem}\label{lem-LG1}
Keep the notations above.  For each indecomposable object \(X\) in
\(\mathcal{C}_d(k\Delta)\) not fixed by \(G\), the orbit of \(X\) under the
action of \(G\) has size divisible by \(L_G\). 
\end{Lem}
\begin{proof}
For type $\mathbb{A}_n$, we use the geometric model of
\(\mathcal{C}_d(k\mathbb{A}_n)\) given by \(d\)-diagonals in a regular
\(N\)-gon \(P_N\). The action of \(\eta\) on
\(\mathcal{C}_d(k\mathbb{A}_n)\) is given by a rotation of \(P_N\) by
\(\frac{2\pi\deg\eta}{N}\). Suppose \(X\) corresponds to a \(d\)-diagonal
\(D\).  The case when \(D\) is fixed by \(\eta\) only happens when \(D\) goes
through the center of \(P_N\) and \(\deg\eta=\frac{N}{2}\). In this case, the
size of the orbit of \(X\) under the action of \(G\) is \(1\). If \(D\) is not
fixed by \(\eta\), then the size of the orbit of \(X\) under the action of
\(G\) is exactly \(L_G\), since \(L_G\) is the smallest positive integer such
that \(\deg(\eta)L_G\equiv0\pmod N\).

For type $\mathbb{D}_n$ and $\mathbb{E}_n$, we adopt the notation in \S \ref{subsection-locally-finite}.
Let $x=(i,t)$ be a vertex of $\mathbb{Z}\Delta$ corresponding to $X$ in $\mathcal{C}_d(\Delta)$. The size of the orbit of $X$ under the action of $G$ is exactly the smallest positive integer $m$ such that $\eta^m X = X$. Since $\eta^m X = \psi (X[m\deg\eta])$, where $\psi$ is induced by an automorphism of $\Delta$, this means that $\psi \circ S^{m\deg\eta} x=\tau_{d+1}^r x$ for some $r$. 
Note that  $S=\tau^{-h_{\Delta}^*}\phi$ by \S \ref{subsection-locally-finite}, where $\phi$ is the automorphism of order $2$ for type $\mathbb{D}_n$ with $n$ odd or $\mathbb{E}_6$. Thus, we have 
\(\psi\circ \tau^{-m\deg\eta\cdot h_{\Delta}^*}(i,t)=\tau^{rN}\phi^r(i,t).\)
Note that $\psi$ and $\phi$ do not change the first coordinate of $(i,t)$. It follows that 
\(i+m\deg\eta \cdot h_{\Delta}^* \equiv i \pmod{N}.\)
This implies that $m\deg\eta \cdot h_{\Delta}^* \equiv 0 \pmod{N}$. Note that $(N,h_{\Delta}^*)=1$ since $N=dh_{\Delta}^*+1$. We deduce that $m\deg\eta \equiv 0 \pmod{N}$. Since $L_G$ is the order of $\deg\eta$ in $\mathbb{Z}/N\mathbb{Z}$, we have \(L_G\mid m\).
\end{proof}

\begin{proof}[Necessary part of Theorem \ref{theo-existence-dct} for type $\mathbb{A}$]
Assume that $\T$ has a $d$-cluster tilting object. By Lemma
\ref{lem-dct-H-ctcat}, $\mathcal{C}_d(k\mathbb{A}_n)$ has a $G$-stable
$d$-cluster tilting object. By the geometric model of
$\mathcal{C}_d(k\mathbb{A}_n)$, this is equivalent to a \(G\)-stable
\((d+2)\)-angulation \(W\) of \(P_N\). Note that the action of \(\eta\) on
\(\mathcal{C}_d(k\mathbb{A}_n)\) is given by \([\deg\eta]\), which corresponds
to a rotation of \(P_N\) by \(\frac{2\pi\deg\eta}{N}\).

If $W$ contains a $d$-diagonal through the center of $P_N$, then the diagonal $D$ must be of the form $D_{(i,i+\frac{N}{2})}$. It follows that 
\(\frac{N}{2}-1\in d\mathbb{Z}\). 
This means that 
\(\frac{n+1}{2}d\in d\mathbb{Z}\). 
Hence $n+1$ is even, or equivalently, $n$ is odd. In this case  $\eta(D)=D$ and $\frac{2\pi\deg\eta}{N}$ must be a multiple of $\pi$. It follows that $L_G=1$ or $L_G=2$.  If $L_G=1$, then clearly $L_G \mid (n,d+2)$. 

Now assume that \(W\) does not contain any \(d\)-diagonal through the center of
\(P_N\). Let \(C\) be the unique \((d+2)\)-gon in \(W\) with the center of
\(P_N\) inside it. It follows that \(\eta(C)=C\). Thus \(G\) acts on the set
of \(d+2\) edges of \(C\). In this case, \(G\) has no fixed edge. It follows
from Lemma \ref{lem-LG1} that each orbit has size divisible by \(L_G\). Hence
\(L_G\mid d+2\). Similarly \(G\) acts on the \(n\) \(d\)-diagonals in \(W\),
and \(L_G\mid n\). Thus \(L_G\mid(d+2,n)\). This proves the necessary
conditions in type \(\mathbb A\).
\end{proof}

\begin{proof}[Necessary part of Theorem \ref{theo-existence-dct} for type $\mathbb{D}$]
From Table \ref{tab:existence-conditions}, we can assume that  
\(L_G>1\), since $L_G=1$ occurs in the conditions for each case.

Let first $n\neq 4$ or let the automorphism part have order at most two. Thus
\(\eta=\tau^r\psi,\ \psi=\id \text{ or } \phi,\)
where $\phi$ exchanges the two branch vertices of $D_n$. Since
$\deg(\tau)=d$ and $\deg(\phi)=0$, the action induced by $\eta$ on the
geometric model of $\mathcal C_d(k\mathbb D_n)$ is the following. It rotates
$P_N^\circ$ by $\frac{2\pi rd}{N}$, and on tagged radii the tag is changed when $rd$ is odd (see \S\ref{subsec-geometric-model}); the automorphism $\phi$ changes
the tag once more.

For simplicity, put
\(L=L_G,\ g=(N,r),\ q=\frac{rL}{N}=\frac rg .\)
Since $(d,N)=1$, the order of $rd$ in $\mathbb Z/N\mathbb Z$ is the same as
the order of $r$; hence $L=N/g$ and $q$ is an integer coprime to $L$.

Suppose $\T$ has a $d$-cluster tilting object. Then there is a  $G$-stable $(d+2)$-angulation $W$ of
$P_N^\circ$. Equivalently, $W$ is the set of indecomposable summands of a
$G$-stable basic $d$-cluster tilting object in $\mathcal C_d(k\mathbb D_n)$. By definition, evidently every $(d+2)$-angulation of the punctured polygon contains a tagged radius.
Moreover, if $L>1$, no indecomposable object of $\mathcal C_d(k\mathbb D_n)$
is fixed by $G$. Thus all $G$-orbits occurring in \(W\) have cardinality
divisible by \(L\) by Lemma \ref{lem-LG1}. Since a basic \(d\)-cluster tilting
object of type \(D_n\) has exactly \(n\) indecomposable summands, we obtain
\(L\mid n\). Write
\(a=\frac nL.\)
Since $L\mid N$ and $L\mid n$, we also have
\(0\equiv N=(n-1)d+1\equiv 1-d \pmod L,\)
and therefore
\(d=Lb+1\)
for some integer $b\geq 0$. It follows that
\(g=\frac NL=b(n-1)+a.\)

We next record the extra restriction coming from tagged radii. Since \(W\)
contains a tagged radius and the endpoint rotation has order \(L\), the
\(G\)-orbit of this tagged radius must have exactly \(L\) elements. Indeed, if
the \(L\)-th iterate returned to the same endpoint with the opposite tag, then
the orbit would contain both tagged radii over one endpoint. Also, by Lemma \ref{lem-non-crossing}(3), for
distinct endpoints \(0\leq i\neq j<N\), the tagged radius \(D_{ii}^{+}\) is
\(d\)-orthogonal to precisely one of \(D_{jj}^{+}\) and \(D_{jj}^{-}\). Hence a
tagged-radius orbit occurring in \(W\) cannot have twice the endpoint orbit since we assume that $L>1$. Thus 
\(\eta^L(D_{00}^+)=D_{00}^+.\)
Define $\epsilon(\psi)$ to be $1$ when $\psi=\id$ and $-1$ when $\psi=\phi$. Then the tag changing rule  implies that  
\[(-1)^{rdL}\epsilon(\psi)^L=(-1)^{dNq}\epsilon(\psi)^L=1.\tag{$\star$}\]
The \(L\)-point orbit of a tagged radius in $W$ must be pairwise
\(d\)-orthogonal. By applying a power of \(\tau\), it is enough to test the
orbit of \(D_{00}^{+}\). For \(1\leq k\leq L-1\), let \(q_k\) be the
representative of \(-kq\) modulo \(L\) in \(\{1,\ldots,L-1\}\), and put
\(x_k=q_k g,\ m_k=\frac{kq+q_k}{L}.\)
Then the \(k\)-th non-trivial endpoint in the orbit is represented by \(x_k\),
and
\(kr=-x_k+m_kN.\)
Thus \(m_k\) records the number of full \(N\)-turns needed to write the orbit
point in the form \(\tau^{-x_k}(D_{00}^{\epsilon})\). Moreover, since
\(x_k=q_kg=q_kb(n-1)+q_ka\)
and
\(0<q_ka\leq (L-1)a=n-a\leq n-1,\)
we have
\(\left\lceil\frac{x_k}{n-1}\right\rceil=q_kb+1.\)
It follows that
\[
\begin{aligned}
x_k+\left\lceil\frac{x_k}{n-1}\right\rceil+1
&=q_kb(n-1)+q_ka+q_kb+1+1\\
&=q_kbn+q_ka+2\\
&=q_ka(Lb+1)+2\\
&=q_kad+2.
\end{aligned}
\]
Here we used \(n=La\) and \(d=Lb+1\) in the third equality. Thus the exponent
appearing in Lemma \ref{lem-non-crossing}(2) satisfies
\[
x_k+\left\lceil\frac{x_k}{n-1}\right\rceil+1
\equiv q_kad\pmod2.
\tag{$*$}
\]
Since
\(\tau^{m_kN}\) changes the tag by \((-1)^{dm_kN}\), we have
\(\eta^k(D_{00}^{+}) =\tau^{-x_k}\bigl(D_{00}^{(-1)^{dm_kN}\epsilon(\psi)^k}\bigr).\)
Write \(e(\psi)=0\) when \(\psi=\id\) and \(e(\psi)=1\) when
\(\psi=\phi\).  By Lemma \ref{lem-non-crossing}(2) and \((*)\), the
\(k\)-th orbit point is \(d\)-orthogonal to \(D_{00}^{+}\) if and only if
\[
dm_kN+e(\psi)k\equiv q_kad\pmod2.
\tag{$**$}
\]
Thus we have shown that the $G$-orbit of $D_{00}^+$ is pairwise $d$-orthogonal if and only if the condition $(\star)$ holds and $(**)$ holds for all $k=0,1,\cdots,L$. 

If \(\psi=\id\), then $e(\psi)=0$ and $\epsilon(\psi)=1$. Since we have proved that $L\mid n$, from Table \ref{tab:existence-conditions}, there are only two remaining cases. The conditions \((**)\) and  \((\star)\), is equivalent to:
\[
\begin{array}{ll}
 d \text{ even}: & L\mid n,\\
n \text{ even},\, d \text{ odd}:&
a=n/L\text{ is even},\\
n,d \text{ odd}:&
\text{\(q=\frac{rL}{N}\) is even},\\
\end{array}
\]
Indeed, if $d$ is even, then  $(\star)$ and $(**)$ hold for all $k$ automatically in this case.   Other than $L\mid n$, no further conditions are needed. This proved the first and the fourth cases. In the second case \(N\) is even, so the left hand
side of \((**)\) is zero and the condition is \(q_ka\equiv0\pmod2\) for every
\(k\).  If \(a\) is odd, then \(n=La\) and \(n\) even force \(L\) to be
even. Since \(q\) is coprime to \(L\), \(q\) is odd. The defining congruence
\(q_k\equiv-kq\pmod L\) may therefore be reduced modulo \(2\), giving
\(q_k\equiv k\pmod2\). Hence \(0\equiv q_ka\equiv q_k\equiv k\pmod2\) for all \(k\). This is impossible for $L>1$. Thus $a$ is even. When $a$ is even, since $N$ is even in this case, $(**)$ holds for all $k$ and $(\star)$ also holds.  In the third case \(L,a,d,N\) are odd, and
\((**)\) is equivalent to \(kq\equiv0\pmod2\) for every \(k\); together with
\((\star)\), this is precisely the condition that \(q\) is even.  

If \(\psi=\phi\), then $e(\psi)=1$, and the extra factor \(\epsilon(\psi)^k=(-1)^k\) gives:
\[
\begin{array}{ll}
d \text{ even}:&
\text{impossible for }L>1,\\
n \text{ even},\, d \text{ odd}:&
a=n/L\text{ is odd},\\
n,d \text{ odd}:&
\text{\(q=\frac{rL}{N}\) is odd},\\
\end{array}
\]
For $d$ even, \((**)\) would require \(k\equiv0\pmod2\) for all \(k\), so
it fails when \(L>1\). This proves the first case.  For \(n\) even and \(d\) odd, \(N\) is even and
\((**)\) becomes \(k\equiv q_ka\pmod2\). If \(a\) is even, then
\(q_ka\equiv0\pmod2\) for every \(k\), so the congruence already fails for
\(k=1\).  If \(a\) is odd, then \(n=La\) and \(n\) even force \(L\) to be
even. Since \(q\) is coprime to \(L\), \(q\) is odd. The defining congruence
\(q_k\equiv-kq\pmod L\) may therefore be reduced modulo \(2\), giving
\(q_k\equiv k\pmod2\). Hence \(q_ka\equiv q_k\equiv k\pmod2\) for all \(k\).
Thus \((**)\) holds for all \(k\) precisely when \(a\) is odd, and in this case $L$ and $N$ are even, thus $(\star)$ holds.  For \(n,d\) odd, the numbers \(L,a,d,N\) are all odd. Thus \((**)\) becomes
\(m_k+k\equiv q_k\pmod2\). Since \(m_kL=kq+q_k\) and \(L\) is odd, we have
\(m_k\equiv kq+q_k\pmod2\). Substitution gives
\(kq+q_k+k\equiv q_k\pmod2\), or equivalently
\(k(q+1)\equiv0\pmod2\) for all \(k\). Together with
\((\star)\), which in this case reads \((-1)^q(-1)^L=1\) and hence also
forces \(q\) to be odd, this is precisely the condition that \(q\) is odd.

This proves the necessary conditions for the order at most two automorphism
cases.

It remains to discuss the order-three automorphism in type $D_4$. If
$G=\langle\tau^r\phi_3\rangle$ with $o(\phi_3)=3$, then the $G$-orbit of a
branch object has length $\operatorname{lcm}(L_G,3)$. If $L_G>1$, Lemma
\ref{lem-LG1} forces $L_G\mid4$, so $L_G=2$ or $4$; but then
$\operatorname{lcm}(L_G,3)>4$, impossible for a basic $d$-cluster tilting
object of type $D_4$, which has only four indecomposable summands. Hence
$L_G=1$. This proves the necessary condition for the order-three case.
\end{proof}

\begin{proof}[Necessary part of Theorem \ref{theo-existence-dct} for type $\mathbb{E}$]
We work directly on the AR-quiver of \(\mathcal C_d(k\mathbb E_n)\), using the
labelling of the Dynkin diagram fixed in \S \ref{subsection-locally-finite}.
Let
\(h^*=h_{\mathbb E_n}^*\)
be the half Coxeter number, so \(h_{\mathbb E_6}^*=6\),
\(h_{\mathbb E_7}^*=9\), and \(h_{\mathbb E_8}^*=15\).  Write
\(N=N_d(\mathbb E_n)=dh^*+1.\)
The indecomposable objects are represented by vertices
\((i,j),\ i\in\mathbb Z/N\mathbb Z,\ 1\leq j\leq n,\)
and the AR-translation is \(\tau(i,j)=(i-1,j)\).  The shift is given by
\([1]=\tau^{-h^*}\phi,\)
where \(\phi\) is the non-trivial graph automorphism of \(\mathbb E_6\) and the
identity for \(\mathbb E_7\) and \(\mathbb E_8\).

We recall the finite calculation used below.  For a vertex \(x\), let
\(H^-(x)\) be the backward Hom-hammock, i.e. the set of vertices \(y\) with
\(\Hom(y,x)\neq0\).  It is obtained by the standard knitting recurrence.  For a
vertex \(x\), define
\[
B(x)=\bigcup_{t=1}^d H^-(x[t]),\quad
O(x)=\bigl(\mathbb Z/N\mathbb Z\times(\mathbb E_n)_0\bigr)\setminus B(x).
\]
Two vertices \(x,y\) are \(d\)-orthogonal if and only if either
\(y\in O(x)\) or \(x\in O(y)\).  

We now impose \(G\)-stability.  Suppose first that \(L=L_G>1\).  By Lemma
\ref{lem-LG1}, every \(G\)-orbit occurring in a \(G\)-stable basic
\(d\)-cluster tilting object has cardinality divisible by \(L\).  Since a
basic \(d\)-cluster tilting object in type \(\mathbb E_n\) has exactly \(n\)
indecomposable summands, we must have
\(L\mid n\). Also \(L\mid N\) by definition.  In type \(\mathbb E_6\), we have
\(n=6\) and \(N=6d+1\), so the only common divisor is \(1\).  Therefore
\(L=1\) in type \(\mathbb E_6\).  In type \(\mathbb E_7\), one obtains
\(L=1\) or \(7\).  In type \(\mathbb E_8\), one obtains
\(L\in\{1,2,4,8\}\).

It
remains to exclude the possible value \(L=8\) in type \(\mathbb E_8\).  We
 record the non-trivial \(L>1\) orbit calculation for
\(\mathbb E_7\) and \(\mathbb E_8\), which will also be used in the counting
theorem below. Write the generator as \(\tau^r\).
Since \((d,N)=1\), \(r\) and \(rd\) have the same order in
\(\mathbb Z/N\mathbb Z\). Since \(L\) is the order of \(rd\), putting \(a=N/L\)
gives \(\langle\tau^r\rangle=\langle\tau^a\rangle\).  We write
	\[
	\mathcal O_{s,j}^{(L)}
	=
	\{(s-ka,j)\mid k=0,\cdots,L-1\},
	\quad s\in\mathbb Z/a\mathbb Z .
	\]
	For the internal orthogonality test, we record
	the bad residues
	\[F_L(j) = \{m\in\mathbb Z/L\mathbb Z\mid (ma,j)\in B((0,j))\}.\]
The orbit $\mathcal O_{s,j}^{(L)}$ is internally pairwise $d$-orthogonal if and only if $F_L(j)=\varnothing$.

		We now spell out the membership test in \(B((0,j))\). Since we are now in
		type \(\mathbb E_7\) or \(\mathbb E_8\), the graph automorphism part is
		trivial, and
		\((0,j)[t]=(h^*t,j)\).  Define the same-level hammock residues
		\(C_j=\{b\in\{0,\ldots,h^*-1\}\mid (-b,j)\in H^-((0,j))\}.\)
The knitting recurrence gives the following \(C_j\)'s:
		\begin{center}
	\small
	\begin{tabular}{c|c|c}
	\toprule
	type&\(j\)&\(C_j\)\\
	\midrule
		\(\mathbb E_7\)&\(1\)&\(\{0,3,5,8\}\)\\
		\(\mathbb E_7\)&\(2,3,5\)&\(\{0,1,2,3,4,5,6,7,8\}\)\\
		\(\mathbb E_7\)&\(4\)&\(\{0,2,3,4,5,6,8\}\)\\
		\(\mathbb E_7\)&\(6\)&\(\{0,1,3,4,5,7,8\}\)\\
		\(\mathbb E_7\)&\(7\)&\(\{0,4,8\}\)\\
		\midrule
		\(\mathbb E_8\)&\(1\)&\(\{0,3,5,6,8,9,11,14\}\)\\
		\(\mathbb E_8\)&\(2,3,5,6\)&\(\{0,1,\ldots,14\}\)\\
		\(\mathbb E_8\)&\(4\)&\(\{0,2,3,4,5,6,7,8,9,10,11,12,14\}\)\\
		\(\mathbb E_8\)&\(7\)&\(\{0,1,4,5,6,8,9,10,13,14\}\)\\
		\(\mathbb E_8\)&\(8\)&\(\{0,5,9,14\}\)\\
		\bottomrule
		\end{tabular}
\vspace{0.5em}
\captionof{table}{Same-level Hom-hammock residues.}
\label{tab:exceptional-Cj}
		\end{center}
Since \(N=h^*d+1\), the integers \(h^*t-b\), with \(1\leq t\leq d\) and
		\(0\leq b\leq h^*-1\), lie between \(1\) and \(N-1\).  Therefore there is no
		wrap-around ambiguity, and
		\(
		(ma,j)\in B((0,j))\) if and only if \(ma=h^*t-b\) for some $1\leq t\leq d$ and $b\in C_j$.  
		Equivalently, if \(b_m\) denotes the unique residue in
		\(\{0,\ldots,h^*-1\}\) satisfying
		\(b_m\equiv -ma\pmod {h^*}\), 
		then $m\in F_L(j)$ if and only if \(b_m\in C_j\). Next impose \(L\mid N=h^*d+1\).  Write \(d=\omega L+l\), where
		\(0\leq l<L\).  Then \(l\) is the unique residue with
		\(L\mid lh^*+1\).  Thus 
			\(a=\frac{h^*d+1}{L} =h^*\omega+\frac{lh^*+1}{L}.\)
Write $\beta=\frac{lh^*+1}{L}$ for simplicity.  Then $b_m\equiv -ma\equiv -m\beta \pmod{h^*}$. Note that $l$ and $\beta$ can be easily computed once $L$ is known, as is shown in the following small table:
\begin{center}
\[
		\begin{array}{c|c|c|c|c}
	\toprule
	\text{case}&l&\beta&m&b_m\\
	\hline
	\mathbb E_7,\ L=7&3&4&1,2,3,4,5,6&5,1,6,2,7,3\\
	\mathbb E_8,\ L=2&1&8&1&7\\
	\mathbb E_8,\ L=4&1&4&1,2,3&11,7,3\\
	\mathbb E_8,\ L=8&1&2&1,2,3,4,5,6,7&13,11,9,7,5,3,1\\
	\bottomrule
		\end{array}
		\]
\captionof{table}{Residues \(b_m\).}
\label{tab:exceptional-bm}
\end{center}
Thus \(F_L(j)\) is obtained by keeping precisely those \(m\)'s whose displayed \(b_m\) belongs to \(C_j\).  The resulting tables are as follows. 
The resulting one-sided bad-residue tables are:
\begin{center}
	\[
	\begin{array}{c|ccccccc}
	\multicolumn{8}{c}{\mathbb E_7,\ L=7}\\
    \toprule
		j&1&2&3&4&5&6&7\\
		\hline
		F_7(j)&
		\{1,6\}&\{1,2,3,4,5,6\}&\{1,2,3,4,5,6\}&
		\{1,3,4,6\}&\{1,2,3,4,5,6\}&\{1,2,5,6\}&\varnothing\\
		\bottomrule
	\end{array}
	\]
\captionof{table}{The bad-residue sets \(F_7(j)\).}
\label{tab:E7-F7}
\end{center}
	and, for \(\mathbb E_8\),
	\begin{center}
	\small
	\begin{tabular}{c|c|c}
	\toprule
	\(L\)&\(j\)&\(F_L(j)\)\\
	\midrule
		\(2\)&\(1,7,8\)&\(\varnothing\)\\
		\(2\)&\(2,3,4,5,6\)&\(\{1\}\)\\
		\midrule
		\(4\)&\(1\)&\(\{1,3\}\)\\
		\(4\)&\(2,3,4,5,6\)&\(\{1,2,3\}\)\\
		\(4\)&\(7,8\)&\(\varnothing\)\\
		\midrule
		\(8\)&\(1\)&\(\{2,3,5,6\}\)\\
		\(8\)&\(2,3,5,6\)&\(\{1,2,3,4,5,6,7\}\)\\
		\(8\)&\(4\)&\(\{2,3,4,5,6\}\)\\
		\(8\)&\(7\)&\(\{1,3,5,7\}\)\\
		\(8\)&\(8\)&\(\{3,5\}\)\\
	\bottomrule
	\end{tabular}
\vspace{0.5em}
\captionof{table}{The bad-residue sets \(F_L(j)\) for type \(\mathbb E_8\).}
\label{tab:E8-FL}
	\end{center}
For \(\mathbb E_7\) and \(L=7\), the table shows that the only internally
orthogonal level is \(j=7\); this is compatible with the necessary condition
\(L=7\).  For \(\mathbb E_8\), the case \(L=8\) cannot happen, since all
\(F_8(j)\) are non-empty.  Indeed, a basic object has eight indecomposable
summands and every \(G\)-orbit has cardinality divisible by \(8\), so a
\(G\)-stable object would have to be a single internally pairwise
\(d\)-orthogonal orbit.  The table excludes such an orbit.  Hence only
\(L=1,2,4\) remain in type \(\mathbb E_8\).  This proves the necessary
conditions in  type $\mathbb{E}$.
\end{proof}

\subsection{Sufficiency and counting}\label{subsection-counting}

In this subsection, we show that the conditions in Table \ref{tab:existence-conditions} are also sufficient for the existence of $d$-cluster tilting objects, and give a counting theorem for the number of basic $d$-cluster tilting objects. We first recall the numerical input used in the counting formulas.  For a
finite Coxeter diagram \(\Omega\), let \(h_{\Omega}\) be its Coxeter number and
let \(d_1,\ldots,d_m\) be the degrees of the corresponding Coxeter group,  which can be
found, for instance, in \cite[Chapter 3]{Humphreys1990}. By Fomin--Reading \cite{FominReading2005}, the
\(d\)-Fuss--Catalan number  
\[\operatorname{Cat}^{(d)}(\Omega)=\prod_{i=1}^m\frac{dh_{\Omega}+d_i}{d_i}.\]
is the number of facets of the
generalized cluster complex of type \(\Omega\).  For Dynkin type, using the
cluster-combinatorial model of \(d\)-cluster categories in
\cite[Corollary 5.8]{Zhu2008}, it is also the number of basic \(d\)-cluster tilting objects in the corresponding $d$-cluster category.
For \(a\geq1\), write
\[C_a^{(d+1)}=\frac{1}{da+1}\binom{(d+1)a}{a}.\]
The following theorem records the corresponding counting formulas.  It also
proves the sufficiency part of Theorem \ref{theo-existence-dct}: the listed
formulas are positive in the cases allowed by Table
\ref{tab:existence-conditions}.
Throughout this subsection we count labelled \(G\)-stable basic
\(d\)-cluster tilting objects in \(\mathcal C_d(k\Delta)\).  We do not divide
by the action of \(G\).

\begin{Theo}\label{theo-counting-dct}
Let \(\Delta\) be a Dynkin diagram, and keep the notation of Theorem
\ref{theo-existence-dct}. Let \(Q(\Delta,d,L;\psi)\) be the number of labelled
basic \(G\)-stable \(d\)-cluster tilting objects in
\(\mathcal{C}_d(k\Delta)\), where \(\psi\) is the diagram automorphism part of
the generator of \(G\) and \(L=L_G\).

\begin{enumerate}
\item In type \(\mathbb A_n\), we have
\[
Q(\mathbb{A}_n,d,L)=
\begin{cases} 
    C_{n+1}^{(d+1)},& L=1,\\
\binom{\frac{(d+1)n+d+2}{L}-1}{\frac nL},
& L\geq2,\ L\mid n,\ L\mid(d+2),\\
\binom{\frac{(d+1)(n+1)}2}{\frac{n+1}2}, & L=2,\ n\text{ is odd},\\
0,& \text{otherwise.}
\end{cases}
\]
\item In type \(\mathbb D_n\), write \(\eta=\tau^r\psi\).  For \(L=1\),
\[
Q(\mathbb{D}_n,d,1;\psi)=
\begin{cases}
\operatorname{Cat}^{(d)}(\mathbb D_n),& \psi=\id,\\
\operatorname{Cat}^{(d)}(\mathbb B_{n-1}),&
\psi=\phi\text{ of order }2,\\
\operatorname{Cat}^{(d)}(\mathbb G_2),&
n=4,\ \psi=\phi_3\text{ of order }3.
\end{cases}
\]
If \(L=L_G>1\) and the corresponding row of Theorem
\ref{theo-existence-dct} is satisfied, then, with
\(a=n/L\) and \(N=N_d(\mathbb D_n)=(n-1)d+1\),
\[
Q(\mathbb{D}_n,d,L;\psi)=
\frac{2N}{L}\, C_a^{(d+1)}
=
\frac{2N}{L}\cdot
\frac{1}{da+1}\binom{(d+1)a}{a}.
\]

\item For type $\mathbb{E}$, the \(L=1\) numbers are
\[
Q(\mathbb{E}_n,d,1;\psi)=
\begin{cases}
\operatorname{Cat}^{(d)}(\mathbb E_n),& \psi=\id,\\
\operatorname{Cat}^{(d)}(\mathbb F_4),&
n=6,\ \psi=\phi\text{ is the non-trivial graph automorphism}.
\end{cases}
\]
For \(\mathbb E_7\),
\(Q(\mathbb E_7,d,7;\id)=\frac{9d+1}{7}\)
whenever \(7\mid 9d+1\).  For \(\mathbb E_8\),
\(Q(\mathbb E_8,d,2;\id)=((d+1)(15d+1)(5d+3)(15d+7))/64\)
when \(d\) is odd, and
\(Q(\mathbb E_8,d,4;\id)=\frac{(15d+1)(5d+3)}{16}\)
when \(d\equiv1\pmod4\).
\end{enumerate}
\end{Theo}

\begin{proof}
In the \(L=1\) cases with a non-trivial graph automorphism, the fixed
\(d\)-cluster tilting objects are counted by the folded types
\(
\mathbb D_n/\langle\phi\rangle=\mathbb B_{n-1},
\mathbb D_4/\langle\phi_3\rangle=\mathbb G_2\), and \(\mathbb E_6/\langle\phi\rangle=\mathbb F_4.
\)

We first discuss type \(\mathbb A_n\).  For \(L=1\), the count is the ordinary
Fuss--Catalan number \(C_{n+1}^{(d+1)}\).  Suppose \(L=2\) and \(n\) is odd.
Every \(G\)-stable \((d+2)\)-angulation contains a unique diameter: without a
diameter the necessary argument above gives \(2\mid n\), and two distinct
diameters cross.
Put \(a=(n+1)/2\).  There are \(N/2=ad+1\) diameters.  After choosing one of
them, the \(G\)-stable angulation is determined by a \((d+2)\)-angulation of one
of the two \((ad+2)\)-gons.  Hence the count is
\((ad+1)C_a^{(d+1)} = \binom{(d+1)a}{a} = \binom{\frac{(d+1)(n+1)}2}{\frac{n+1}2}.\)

Now suppose \(L\geq2\), \(L\mid n\), and \(L\mid(d+2)\).  Write
\(d+2=mL, n=qL\), 
and put \(M=N/L=dq+m\).  A \(G\)-stable angulation in this case has no diameter
and has a unique central \((d+2)\)-gon \(C\).  After replacing the generator by
a coprime power, we may assume that the rotation is \(i\mapsto i+M\).  Choose
one vertex orbit of \(C\) as a cut, and first normalize it to be represented by
the vertices \(0\) and \(M\).  In the interval from \(0\) to \(M\), the
representatives of the \(m=(d+2)/L\) side orbits of \(C\) have endpoints
\(0=i_0<i_1<\cdots<i_{m-1}<i_m=M\).  The whole central polygon has vertex set
\(\{i_\beta+kM\mid 0\leq\beta<m,\ 0\leq k<L\}\), where \(i_m\) is identified
with \(i_0+M\).

For \(1\leq j\leq m\), the side \((i_{j-1},i_j)\) of \(C\) is a \(d\)-diagonal.
Hence there is a unique non-negative integer \(s_j\) such that
\(i_j-i_{j-1}=ds_j+1\).
Summing over \(j\) gives
\(s_1+\cdots+s_m=(M-m)/d=q\), and conversely any such \(m\)-tuple determines
the vertices by \(i_j=j+d(s_1+\cdots+s_j)\).  The side
\((i_{j-1},i_j)\), together with the boundary interval from \(i_{j-1}\) to
\(i_j\), bounds a polygon \(P_j\) with \(ds_j+2\) vertices.  Thus \(P_j\) is
filled by a rooted \((d+2)\)-angulation counted by \(t_{s_j}\), where \(t_s\)
denotes the number of rooted \((d+2)\)-angulations of a polygon with \(ds+2\)
vertices.  Therefore, for this fixed cut, the count is
\(\sum_{s_1+\cdots+s_m=q}t_{s_1}\cdots t_{s_m}\). Let \(T(z)=\sum_{s\geq0}t_s z^s\).  The term \(t_0=1\) represents the empty
piece attached to a side.  If the rooted piece is non-empty, the unique
\((d+2)\)-gon adjacent to the root side contributes one factor \(z\), and its
other \(d+1\) sides carry \(d+1\) independent ordered rooted pieces.  Hence
\(T(z)=1+zT(z)^{d+1}\).  The above sum over all
\(s_1+\cdots+s_m=q\) is exactly the coefficient of $z^q$ in \(T(z)^m\), denoted by $[z^q]T(z)^m$. There are \(M=dq+m\) possible choices for the representative of the cut vertex
in the quotient polygon, and each labelled \(G\)-stable angulation is counted
\(m\) times, once for each vertex orbit of its central polygon.  Therefore the
desired count is \(\frac{M}{m}[z^q]T(z)^m\).  Put \(U=T-1\).  Then
\(U=z(1+U)^{d+1}\), and Lagrange inversion (see, for example,
\cite{Gessel2016,SuryaWarnke2023}) gives
\[[z^q]T(z)^m=[z^q](1+U)^m=\frac{m}{q}[u^{q-1}](1+u)^{(d+1)q+m-1}.\]
Hence
\([z^q]T(z)^m=\frac{m}{(d+1)q+m}\binom{(d+1)q+m}{q}
=\frac{m}{dq+m}\binom{(d+1)q+m-1}{q}.\)
Thus the count is
\[
\frac{dq+m}{m}\cdot
\frac{m}{dq+m}\binom{(d+1)q+m-1}{q}
=
\binom{(d+1)q+m-1}{q}
=
\binom{\frac{(d+1)n+d+2}{L}-1}{\frac nL}.
\]

For type \(\mathbb D_n\) with \(L>1\), it follows from the proof of the necessary part of Theorem \ref{theo-existence-dct} for type $\mathbb{D}_n$ that the conditions listed in Theorem \ref{theo-existence-dct} are sufficient and necessary conditions for the $G$-orbit of $D_{00}^{+}$ to be pairwise $d$-orthogonal and to have size $L$. For such an orbit, to prove the existence of $d$-cluster tilting objects, it is enough to extend the $G$-orbit of tagged radii to a $(d+2)$-angulation of $P_N^{\circ}$. The
\(G\)-orbit of tagged radii has \(L\) members and divides the punctured polygon
into \(L\) sectors.  There are \(2N/L\) choices for this tagged-radius orbit.
After choosing the filling of one sector by a type \(\mathbb A_{a-1}\)
\(d\)-cluster tilting object, where \(a=n/L\), \(G\)-stability determines the
fillings of all remaining sectors.  This gives the displayed formula.
All remaining type \(\mathbb A\) and type \(\mathbb D\) cases have count zero
by the necessary part proved above.

Finally consider type $\mathbb{E}$.  The non-trivial \(L>1\) cases are
precisely the orbit-graph cases computed in the exceptional necessary
calculation above.  For \(\mathbb E_7,L=7\), the objects are
exactly
\(
\{(s-ka,7)\mid k=0,\ldots,6\},  0\leq s<(9d+1)/7,
\)
so the count is \((9d+1)/7\).

It remains to consider $L=2,4$ for $\mathbb{E}_8$.  Put \(h^*=15\),
\(N=15d+1\), and \(a=N/L\).  We use the following orbit graph.  Its vertices
are the internally orthogonal \(G\)-orbits \(\mathcal O_{s,j}^{(L)}\), where
\(s\in \mathbb Z/a\mathbb Z\).  Two such vertices are joined when every object
in one orbit is \(d\)-orthogonal to every object in the other orbit.  A
\(G\)-stable \(d\)-cluster tilting object in type \(\mathbb E_8\) is then the
same thing as a clique of size \(8/L\) in this orbit graph. Let
\(C_{i,j}=\{0\leq b<h^*\mid (-b,i)\in H^-((0,j))\},\quad i,j=1,7,8.\)
The Hom-hammock calculation gives the following \(3\times3\) matrix, whose rows
are indexed by \(i\in\{1,7,8\}\) and columns by \(j\in\{1,7,8\}\):
\begin{center}
\[
\renewcommand{\arraystretch}{1.15}
\begin{array}{c|ccc}
\toprule
C_{i,j}
& j=1 & j=7 & j=8\\
\hline
i=1&
\{0,3,5,6,8,9,11,14\}&
\{2,3,5,6,7,8,10,11\}&
\{2,5,7,10\}\\
i=7&
\{3,4,6,7,8,9,11,12\}&
\{0,1,4,5,6,8,9,10,13,14\}&
\{0,4,5,8,9,13\}\\
i=8&
\{4,7,9,12\}&
\{1,5,6,9,10,14\}&
\{0,5,9,14\}\\
\bottomrule
\end{array}
\]
\captionof{table}{The Hom-hammock residue matrix.}
\label{tab:E8-Cij}
\end{center}
We spell out how this matrix is used to construct the orbit graph.   For two orbit vertices
\(\mathcal O_{s,i}^{(L)}\) and \(\mathcal O_{t,j}^{(L)}\), write $\delta=(t-s)_a\in \mathbb{Z}/a\mathbb{Z}$. By applying $\tau^{s}$, the two orbits are $d$-orthogonal if and only if \(\mathcal O_{0,i}^{(L)}\) and \(\mathcal O_{\delta,j}^{(L)}\) are $d$-orthogonal. If $\delta=0$ and $i=j$, then the two orbits coincide. Now we assume that either $\delta>0$ or $\delta=0$ and $i\neq j$ so that we have two distinct orbits.  In a \(d\)-cluster category, the vanishing of \(\Hom(X,Y[u])\) for \(1\leq u\leq d\) is equivalent to the vanishing of \(\Hom(Y,X[u])\) for \(1\leq u\leq d\).  Thus it is enough to test one direction.  Since \(0\leq \delta+ma<N\) for all \(0\leq m<L\), the two orbits \(\mathcal O_{0,i}^{(L)}\) and \(\mathcal O_{\delta,j}^{(L)}\) are $d$-orthogonal if and only if 
\((\delta+ma,j)\notin H^-((15u,i)), \quad 1\leq u\leq d, 0\leq m< L.\)
For a term with \(\delta+ma>0\), this is equivalent to 
\(-(\delta+ma)\notin C_{j,i} \pmod{15}.\)
If \(\delta=0\), the term \(m=0\) has horizontal coordinate \(0\) and is not of the form \(15u-b\) with \(1\leq u\leq d\) and \(0\leq b\leq14\); hence this term is omitted in the residue test. From Table \ref{tab:exceptional-bm}, we have \(a\equiv 8\pmod{15}\) for \(L=2\) and \(a\equiv 4\pmod{15}\) for \(L=4\). Hence, whether two orbits $\mathcal O_{s,i}^{(L)}$ and $\mathcal O_{t,j}^{(L)}$ are $d$-orthogonal is completely determined by the residue of $\delta$ modulo $15$, together with the special \(\delta=0\) convention above.

Case (i): \(L=2\).  Here \(d=2p+1\) and \(a=15p+8\).  By the table of
\(F_L(j)\), the possible orbit levels are \(1,7,8\).  The above edge test gives
the following table.  In the row \((i,j)\), with \(i\leq j\), the set \(D_{i,j}\)
contains precisely the residues of positive differences 
\(\delta=(t-s)_a\) for which
the orbits $\mathcal O_{s,i}^{(2)}$ and $\mathcal O_{t,j}^{(2)}$ are $d$-orthogonal; the last column records
whether $\mathcal O_{0,i}^{(2)}$ and $\mathcal O_{0,j}^{(2)}$ are $d$-orthogonal when $i\neq j$. 
\begin{center}
\[
\begin{array}{c|c|c}
\toprule
(i,j)&D_{i,j}&\delta=0\\
\hline
(1,1)&\{3,5\}&\text{-}\\
(1,7)&\{2,5\}&\text{no}\\
(1,8)&\{1,2,4,5,7,9,12,14\}&\text{no}\\
(7,7)&\{4\}&\text{-}\\
(7,8)&\{0,3,4,7,11\}&\text{yes}\\
(8,8)&\{3,4,5,9,11,12,14\}&\text{-}\\
\bottomrule
\end{array}
\]
\captionof{table}{Edge residues \(D_{i,j}\) for \(L=2\).}
\label{tab:E8-L2-Dij}
\end{center}
Let us spell out the first row.  Since \(C_{1,1}=\{0,3,5,6,8,9,11,14\}\)
and \(a\equiv8\pmod {15}\), two $d$-orthogonal orbits
\(\mathcal O_{s,1}^{(2)}\) and \(\mathcal O_{t,1}^{(2)}\), with
\(\delta=(t-s)_a>0\) and \(r\equiv\delta\pmod {15}\), are tested by the two
points \((\delta,1)\) and \((\delta+a,1)\) in
\(\mathcal O_{\delta,1}^{(2)}\).  By the one-direction Hom-hammock criterion
above, these two points must both avoid \(B((0,1))\).  Equivalently,
\(-r, -(r+8)\notin C_{1,1}\pmod{15}\). It is straightforward to check that the forbidden residues are 
\(\{0,1,2,4,6,7,8,9,10,11,12,13,14\}.\)
The allowed residues are therefore \(\{3,5\}\). The other case can be done similarly. 

Although the number of vertices of the orbits graph depends on
\(a=15p+8\), the edge relation depends only on residues modulo \(15\) and on
the special \(\delta=0\) column.  Thus the above table gives a finite way to
find all four-cliques. The process goes as follows.  Partition
\(\{0,\ldots,a-1\}\) by residues modulo \(15\).  
If we define 
\(\mu_r=\{0\leq s<a\mid s\equiv r\pmod {15}\},\)
then
\[|\mu_r|=
\begin{cases}
p+1,&0\leq r\leq7,\\
p,&8\leq r\leq14.
\end{cases}
\]
For each possible level multiset \(J=\{j_1,j_2,j_3,j_4\}\), choose once and
for all an ordered representative
\(j_1\leq j_2\leq j_3\leq j_4 .\)
The finite residue search is then implemented as follows.  Take a residue
tuple
\(R=(r_1,r_2,r_3,r_4)\in(\mathbb Z/15\mathbb Z)^4.\)
Write every actual lift in the form
\[
s_\nu=r_\nu+15q_\nu,\quad
0\leq q_\nu\leq p\text{ if }r_\nu\leq7,\quad
0\leq q_\nu\leq p-1\text{ if }r_\nu\geq8.
\]
Instead of running through all possible \(q_\nu\)'s, enumerate the finitely many
weak orderings of the four actual numbers \(s_1,s_2,s_3,s_4\).  Such a weak
ordering means an ordered partition
\(B_1\mid B_2\mid\cdots\mid B_\ell\)
of the index set \(\{1,2,3,4\}\).  It represents the relations
\[
s_\nu=s_\lambda\quad\text{if }\nu,\lambda\in B_m,
\quad
s_\nu<s_\lambda\quad\text{if }\nu\in B_m,\ \lambda\in B_{m'},\ m<m'.
\]
Equivalently, it is a total ordering of the equality classes of
\(\{s_1,s_2,s_3,s_4\}\), allowing ties inside each class.
Discard a weak ordering if two equal-level vertices lie in the same block,
because we impose the convention
\(s_\nu<s_\lambda\text{ whenever } \nu<\lambda \text{ and }j_\nu=j_\lambda .\)
Also discard it if two indices in the same block have different residues,
since then the equality \(s_\nu=s_\lambda\) is impossible.

For every remaining weak ordering and every pair \(\nu<\lambda\), the residue
of \(\delta_{\nu\lambda}=(s_\lambda-s_\nu)_a\) is known without choosing the
actual \(q_\nu\)'s:
\[
\delta_{\nu\lambda}\equiv
\begin{cases}
0,& s_\nu=s_\lambda,\\
r_\lambda-r_\nu,& s_\nu<s_\lambda,\\
r_\lambda-r_\nu+8,& s_\lambda<s_\nu
\end{cases}
\pmod {15}.
\]
Here the \(+8\) comes from adding \(a\equiv8\pmod {15}\) when
\(s_\lambda-s_\nu<0\).  The pair passes the edge test as follows.  If
\(s_\nu=s_\lambda\), then \(\delta_{\nu\lambda}=0\) and we use the last column
of the table.  If \(s_\nu\neq s_\lambda\), then
\(\delta_{\nu\lambda}>0\), and we require the displayed residue, possibly
\(0\), to belong to \(D_{j_\nu,j_\lambda}\).  A weak ordering is admissible if
all six pairs pass this test.

For an admissible weak ordering, count the integer solutions in the
\(q_\nu\)'s which realize it.  The equalities inside a block give equations
\(q_\nu=q_\lambda\), while an inequality
\(r_\nu+15q_\nu<r_\lambda+15q_\lambda\)
is equivalent to \(q_\nu\leq q_\lambda\) if \(r_\nu<r_\lambda\), and to
\(q_\nu<q_\lambda\) if \(r_\nu\geq r_\lambda\).  Thus each admissible weak
ordering contributes the number of integer points of a fixed finite system of
linear inequalities with bounds \(0\leq q_\nu\leq p\) or \(p-1\).  This number
is obtained by elementary finite sums, hence is a polynomial in \(p\) with degree at most $4$.  Summing
these polynomials over all admissible weak orderings gives the contribution of
the residue tuple \(R\), and summing over all \(15^4\) residue tuples gives the
count for the level multiset \(J\).
 
This finite residue search gives the following
complete list of non-zero level multisets:
\begin{center}
\[
\renewcommand{\arraystretch}{1.15}
\begin{array}{c|c}
\toprule
\text{level multiset}&\text{number of four-cliques}\\
\hline
\{1,1,7,8\}&15p^4+68p^3+107p^2+70p+16\\
\{1,1,8,8\}&75p^4+235p^3+269p^2+133p+24\\
\{1,7,8,8\}&90p^4+288p^3+338p^2+172p+32\\
\{1,8,8,8\}&50p^4+\frac{350}{3}p^3+88p^2+\frac{64}{3}p\\
\{7,7,8,8\}&30p^4+106p^3+138p^2+78p+16\\
\{7,8,8,8\}&15p^4+38p^3+31p^2+8p\\
\{8,8,8,8\}&\frac{25}{4}p^4+\frac{65}{6}p^3+\frac{21}{4}p^2+\frac{2}{3}p.\\
\bottomrule
\end{array}
\] 
\captionof{table}{Four-clique counts for  \(L=2\).}
\label{tab:E8-L2-clique-counts}
\end{center}
Summing the displayed rows gives
\(((p+1)(15p+8)(5p+4)(15p+11))/4\).
Since \(d=2p+1\), this is
\(((d+1)(15d+1)(5d+3)(15d+7))/64\).

Case (ii): \(L=4\).  Here \(d=4p+1\) and \(a=15p+4\).  The possible orbit
levels are \(7\) and \(8\), so we only need to count edges.  The edge table is
\begin{center}
\[
\begin{array}{c|c|c}
\toprule
(i,j)&D_{i,j}&\delta=0\\
\hline
(7,7)&\varnothing&\text{no}\\
(7,8)&\{0,3,7,11\}&\text{yes}\\
(8,8)&\{5,14\}&\text{no}.\\
\bottomrule
\end{array}
\]
\captionof{table}{Edge residues for   \(L=4\).}
\label{tab:E8-L4-Dij}
\end{center}
Thus there are no edges between two distinct level \(7\) orbits.  For a fixed
level \(7\) orbit, the allowed level \(8\) differences are
\(\delta=0,\ 1\leq\delta<a,\ \delta\equiv0,3,7,11\pmod {15}.\)
There are \(4p+2=d+1\) such differences, and hence the number of mixed
\((7,8)\)-edges is \(a(d+1)\).  On level \(8\), the allowed positive
differences have residues \(5\) or \(14\).  There are \(2p\) such differences,
and the two opposite differences define the same unordered edge, so the number
of level \(8\)-edges is \(ap\).  Therefore the total number of edges is
\(a(d+1)+ap=a(5p+2) =\frac{(15d+1)(5d+3)}{16}.\)
The \(L=8\) case has no internally orthogonal orbit, and the remaining
exceptional cases are excluded by the necessary part proved above.  The
non-zero formulas in this theorem are positive exactly in the cases listed in
Table \ref{tab:existence-conditions}; by Lemma \ref{lem-dct-H-ctcat}, they give
\(d\)-cluster tilting objects in \(\T\).  This completes the sufficiency part
of Theorem \ref{theo-existence-dct}.
\end{proof}

\begin{Example}
    For instance, take \(J=(1,1,7,8)\) and \(R=(0,3,5,4)\).  For the basic lift
\((s_1,s_2,s_3,s_4)=(0,3,5,4)\), the six tests read
\(3-0\in D_{1,1},\ 5-0,\ 5-3\in D_{1,7},\ 4-0,\ 4-3\in D_{1,8},\)
and, for the last pair,
\((4-5)_a=4-5+a\equiv 7\pmod {15},\ 7\in D_{7,8}.\)
Hence
\[
\mathcal{O}_{0,1}^{(2)},\quad
\mathcal{O}_{3,1}^{(2)},\quad
\mathcal{O}_{5,7}^{(2)},\quad
\mathcal{O}_{4,8}^{(2)}
\]
form a four-clique.  The corresponding weak ordering is
\(s_1<s_2<s_4<s_3.\)
Since all four residues are between \(0\) and \(7\), the associated inequalities
are
\(0\leq q_1\leq q_2\leq q_4\leq q_3\leq p.\)
This weak-ordering chamber therefore contributes
\(\binom{p+4}{4}\)
cliques.  For this residue tuple, the other admissible weak orderings are
\(s_1<s_4<s_2<s_3,\ s_4<s_1<s_2<s_3.\)
They give respectively
\(0\leq q_1\leq q_4<q_2\leq q_3\leq p,\ 0\leq q_4<q_1\leq q_2\leq q_3\leq p,\)
and hence each contributes
\(\binom{p+3}{4}\)
cliques.  No other weak ordering satisfies the six edge tests.  Therefore the
total contribution of \(J=(1,1,7,8)\) and \(R=(0,3,5,4)\) is
\(\binom{p+4}{4}+2\binom{p+3}{4} =\frac{(p+1)(p+2)(p+3)(3p+4)}{24}.\)
\end{Example}

\subsection{$(d+1)$-Calabi--Yau categories}

In this subsection, we consider locally finite $(d+1)$-Calabi--Yau triangulated categories. This means that the automorphism $\tau_{d+1}$ belongs to $G$. 

\begin{Coro}\label{coro-CY-generator-conditions}
Let \(\T\) be connected locally finite triangulated $k$-category, and suppose that its AR-quiver is
isomorphic to \(\mathbb Z\Delta/G\), where \(G=\langle\eta\rangle\) is non-trivial and weakly
admissible. Assume that \(\T\) is \((d+1)\)-Calabi--Yau. Then Theorem
\ref{theo-existence-dct} gives the following generator conditions for the existence of $d$-cluster tilting objects in \(\T\):
\begin{center}
\renewcommand{\arraystretch}{1.25}
\begin{tabular}{p{0.04\textwidth}|p{0.14\textwidth}|p{0.05\textwidth}|p{0.22\textwidth}|p{0.26\textwidth}}
\toprule
\(\Delta\) & & \(\eta\) &
\((d+1)\)-CY condition & condition for existence\\
\hline
\(\mathbb A_n\)
& \(n\) even
& \(\rho^r\)
& \(r\mid N\)
& \(N/r\mid\gcd(n,d+2)\)\\
\hline
\(\mathbb A_n\)
& \(n\) odd, \(d\) even
& \(\tau^r\)
& \(2r\mid N\)
& \(N/(2r)\mid\gcd(n,d+2)\)\\
\hline
\(\mathbb A_n\)
& \(n\) odd, \(d\) odd
& \(\tau^r\phi\)
& \(2r\mid N\), \(N/(2r)\) odd
& \(N/(2r)\mid\gcd(n,d+2)\)\\
\hline
\(\mathbb D_n\)
& \(n,d\) even
& \(\tau^r\)
& \(r\mid N\)
& \(L\mid n\)\\
\hline
\(\mathbb D_n\)
& \(n\) even, \(d\) odd
& \(\tau^r\)
& \(r\mid N\)
& \(L\mid n\) and \(n/L\) is even\\
\hline
\(\mathbb D_n\)
& \(n\) even, \(d\) odd
& \(\tau^r\phi\)
& \(r\mid N\), \(L\) even
& \(L\mid n\) and \(n/L\) is odd\\
\hline
\(\mathbb D_n\)
& \(n\) odd, \(d\) even
& \(\tau^r\)
& \(r\mid N\)
& \(L\mid n\)\\
\hline
\(\mathbb D_n\)
& \(n\) odd, \(d\) odd
& \(\tau^r\phi\)
& \(r\mid N\), \(L\) odd
& \(L\mid n\)\\
\hline
\(\mathbb E_6\)
& \(d\) even
& \(\tau^r\)
& \(r\mid N\)
& \(L=1\)\\
\hline
\(\mathbb E_6\)
& \(d\) odd
& \(\tau^r\phi\)
& \(r\mid N\), \(L\) odd
& \(L=1\)\\
\hline
\(\mathbb E_7\)
& 
& \(\tau^r\)
& \(r\mid N\)
& \(L\in\{1,7\}\)\\
\hline
\(\mathbb E_8\)
& 
& \(\tau^r\)
& \(r\mid N\)
& \(L\in\{1,2,4\}\)\\
\hline
\multicolumn{5}{c}{\parbox{0.88\textwidth}{\(\mathbb A_n\): if \(n\) is odd,
\(\phi=\tau^{(n+1)/2}S\); if \(n\) is even, \(\rho=\tau^{n/2}S\).
For \(\mathbb D_n\) and \(\mathbb E_6\), \(\phi\) is the distinguished diagram
involution; $L=N/r$.}}\\
\bottomrule
\end{tabular}
\vspace{.5em}
\captionof{table}{$(d+1)$-Calabi--Yau locally finite triangulated categories.}
\label{tab:CY-conditions}
\end{center} 
\end{Coro}

\begin{proof}
The Serre functor on \(\Db[k\Delta]\) is \(\tau[1]\).  Hence the
condition that the quotient \(\T\) is \((d+1)\)-Calabi--Yau means that
\(\tau[1]\simeq[d+1]\) on \(\T\), or equivalently that \(\tau_{d+1}=\tau[-d]\) acts
trivially on the quotient, that is, \(\tau_{d+1}\in G\).

In type \(\mathbb A_n\) with \(n\) even we have \(\tau=\rho^{-2}\) and
\(S=\rho^{n+1}\), hence \(\tau_{d+1}=\tau S^{-d}=\rho^{-N}\).  Therefore
\(\tau_{d+1}\in\langle\rho^r\rangle\) if and only if \(r\mid N\), and then
\(L_G=N/r\).  If \(n\) is odd, then \(N/2\) is an integer and
\(S=\tau^{-(n+1)/2}\phi\).  Hence
\(\tau_{d+1}=\tau S^{-d}=\tau^{N/2}\phi^d\). 
For \(d\) even, \(N/2\) is odd, so the only possible Calabi--Yau generators are
\(\eta=\tau^r\) with \(r\mid N/2\), and then \(L_G=(N/2)/r\).  For \(d\) odd,
a pure power of \(\tau\) cannot contain \(\tau_{d+1}\), while
\(\eta=\tau^r\phi\) contains \(\tau_{d+1}\) only when \(r\mid N/2\) and
\((N/2)/r\) is odd.  In this last case, \(\deg(\phi)=N/2\) in
\(\mathbb Z/N\mathbb Z\), so \(\deg(\eta)=rd+N/2\).  If
\(k=N/(2r)\), then \(k\) is odd and \((d,k)=1\), hence
\[
L_G=\frac{N}{(\deg\eta, N)}=\frac{2rk}{r(d+k,2k)}=k=\frac{N}{2r}.
\]
Thus again \(L_G=N/(2r)\).  Substituting these values of \(L_G\) in the type
\(\mathbb A\) row of Theorem
\ref{theo-existence-dct} gives the stated existence conditions.

For type \(\mathbb D\) and \(\mathbb E\), Lemma \ref{lem-N} gives
\(\tau_{d+1}=\tau^N\phi^d\), with \(\phi=\id\) except in the two involutive cases
specified above.  If \(\tau_{d+1}\in\langle\tau^r\psi\rangle\), the \(\tau\)-part
forces \(r\mid N\), and we write \(L=N/r\).  The diagram-automorphism part
is exactly \(\psi^L=\phi^d\).
When \(\phi=\id\), this says \(\psi^L=\id\), which gives the evenness
condition for order-two diagram automorphisms and the divisibility by \(3\) in
the order-three \(D_4\) case.  When \(\phi\neq\id\), the number \(N\), and
hence \(L\), is odd.  Thus for even \(d\) the Calabi--Yau condition leaves
only \(\psi=\id\), while for odd \(d\) it leaves only \(\psi=\phi\).  Since
\((d,N)=1\), in all these cases \(L_G=N/r=L\).  Substituting this value of
\(L_G\) in Theorem \ref{theo-existence-dct} gives the table above.
\end{proof}

\noindent 
{\it Remark.} Burban--Iyama--Keller--Reiten classify in \cite[Theorem A.2]{Burban2008aa}
the connected finite \(2\)-Calabi--Yau triangulated categories which have
ordinary cluster tilting objects.  In their notation the AR-quiver is written
as \(\mathbb Z\Delta/G\), where \(G\) is a weakly admissible cyclic group. When a listed case has
\(G=\langle\tau^g\rangle\), their \(g\) is the exponent of the AR-translation
\(\tau\).  Thus their theorem is the \(d=1\) comparison case for the present
statement.  Translating their generator into the order \(L_G\) of the induced
action on the ordinary cluster category gives the \(d=1\) specialization of
Table \ref{tab:existence-conditions}.

Let us point out that the case for \(\mathbb E_8\) and $G=\langle\tau^4\rangle$ is missing in  \cite[Theorem A.2]{Burban2008aa}. The case
\(\mathbb Z\mathbb E_8/\langle\tau^4\rangle\), that is \(g=4\), is
\(2\)-Calabi--Yau: indeed \([1]=\tau^{-15}\), hence \([1]\simeq\tau\) modulo
\(\tau^4\), and the Serre functor is
\(\tau[1]\simeq\tau^2\simeq[2]\).  Moreover it has a cluster tilting
object, namely the union
\(\{(0,8),(4,8),(8,8),(12,8)\}\cup
\{(0,7),(4,7),(8,7),(12,7)\}\).
The Hom-hammock residues used above show that these eight vertices are
pairwise \(1\)-orthogonal.

\section{Derived equivalences between endomorphism algebras}

The purpose of this section is to turn the existence and mutation results
above into derived equivalences between endomorphism algebras.  We first set
up equivariant mutation and show that a \(G\)-mutation path lifts through the
covering to $\mathcal{D}$-split triangles, and then to \(\mathcal D\)-split sequences in a
Frobenius model.  This gives derived equivalences between the corresponding
endomorphism algebras.  After that, we prove \(G\)-mutation connectedness in
Dynkin type by combining the eventual common-orbit property with
Iyama--Yoshino reduction.

Let \(\mathcal C\) be a triangulated category, and let
\(G\) be a group acting on \(\mathcal C\) by triangle autoequivalences.  
Let \(T\) be a  basic $G$-stable  object in $\mathcal{C}$, and let \(T=X\oplus W\) be a decomposition such that both \(X\) and \(W\) are
\(G\)-stable.  We say that \(T^*=X^*\oplus W\) is obtained from \(T\) by a
\emph{\(G\)-left mutation at \(X\)} if there is a triangle
\(X\longrightarrow B\longrightarrow X^*\longrightarrow X[1],\)
where \(X\to B\) is a minimal left \(\add(W)\)-approximation.    In this case we call \((T,T^*)\) a
{\em \(G\)-mutation pair}.

Two \(G\)-stable basic \(d\)-cluster tilting objects \(M\) and \(M^*\) are
called \emph{\(G\)-mutation reachable} if there is a sequence of
\(G\)-stable basic \(d\)-cluster tilting objects
\(M=M_0,M_1,\ldots,M_r=M^*\)
such that, for each \(0\le i<r\), either \((M_i,M_{i+1})\) or
\((M_{i+1},M_i)\) is a \(G\)-mutation pair.

\begin{Prop}\label{prop-G-mutation-derived-equivalence}
Assume the covering setting of Section \(3\), and suppose that
\(\tau_{d+1}^m\in G\) for some \(m>0\), as in Lemma
\ref{lem-dct-H-ctcat}.
Suppose moreover that
\(\mathcal T\) is the stable category \(\underline{\mathcal E}\) of a
Frobenius category \(\mathcal E\) with projective generator \(\Lambda\). Let
\(M,M'\) be \(G\)-stable \(d\)-cluster tilting objects in
\(\mathcal C_d(H)\) which are \(G\)-mutation reachable.  Put
\(M_{\mathcal T}=F(\pi^{-1}M)\) and
\(M'_{\mathcal T}=F(\pi^{-1}M')\).  Then the
algebras
$\End_{\mathcal E}(\Lambda\oplus M_{\mathcal T})$ and  
$\End_{\mathcal E}(\Lambda\oplus M'_{\mathcal T})$ 
are derived equivalent.
\end{Prop}

\begin{proof}
By the reachability assumption in \(\mathcal C_d(H)\), it is enough to treat
one \(G\)-mutation step.  Write such a step as
\(M=X\oplus W,\ M^*=Y\oplus W,\)
with \(X,W,Y\) \(G\)-stable.  We use the exchange triangle in the form
\[
X\longrightarrow B\longrightarrow Y\longrightarrow X[1],
\tag{\(\diamond\)}
\]
where the first morphism is a minimal left \(\add(W)\)-approximation.  If the
mutation pair occurs in the reverse order along the chosen path, we reverse
that step; the derived equivalence obtained below is symmetric.

The triangle \((\diamond)\) lies in the \(d\)-cluster category.  We first make
precise the choice of lifts.  For a direct-sum object \(U\) in
\(\mathcal C_d(H)\), write \(\pi^{-1}U\) for the full set of indecomposable
objects in \(\Db[H]\) whose image is an indecomposable summand of \(U\).  If
\(U\) is \(G\)-stable, then \(\pi^{-1}U\) is \(G\)-stable: for
\(\widetilde U\in\pi^{-1}U\) and \(g\in G\), one has
\(\pi(g\widetilde U)=g\pi(\widetilde U)\), which is again a summand of \(U\).
This applies to \(X\) and \(W\).  It also applies to \(B\), because applying
any \(g\in G\) to \((\diamond)\) gives another exchange triangle with a
minimal left \(\add(W)\)-approximation of \(X\), and the middle term is
therefore isomorphic to \(B\).

Moreover, for any finite direct-sum object \(U\) in \(\mathcal C_d(H)\),
the set \(\pi^{-1}U\) is a finite union of
\(\langle\tau_{d+1}\rangle\)-orbits, one for each indecomposable summand of
\(U\).  Put \(K=\langle\tau_{d+1}\rangle\).  If \(U\) is $G$-stable, then
\(\pi^{-1}U\) is $G$-stable.  Since
\(\langle\tau_{d+1}^m\rangle\subseteq G\), each $K$-orbit meets at most
$m$ $G$-orbits.  Hence the fibres \(\pi^{-1}X\), \(\pi^{-1}W\) and
\(\pi^{-1}B\) are finite unions of $G$-orbits.  By the covering assumption
\(F^{-1}(Fz)=Gz\), the
indecomposable summands of \(F(\pi^{-1}U)\) are in bijection with the
\(G\)-orbits contained in \(\pi^{-1}U\).
Choose one representative from each \(G\)-orbit in \(\pi^{-1}X\) and in
\(\pi^{-1}W\), retaining the multiplicities occurring in \(X\) and \(W\).
Denote the resulting finite direct sums by \(X^\sharp\) and \(W^\sharp\),
respectively.  Thus
\(F X^\sharp=F(\pi^{-1}X)\) and
\(F W^\sharp=F(\pi^{-1}W)\).

We now lift the left approximation more explicitly.  For indecomposable
objects \(x,y\in\Db[H]\), one has
\[
\Hom_{\mathcal C_d(H)}(\pi x,\pi y)
 \simeq \bigoplus_{i\in\mathbb Z}
 \Hom_{\Db[H]}(x,\tau_{d+1}^{i}y).
\]
Decompose the minimal left approximation \(X\to B\) into the left
approximations of the indecomposable summands of \(X\).  For every
indecomposable summand \(x\) of \(X^\sharp\), write each matrix entry of the
corresponding approximation as the sum of all its components in the displayed
Hom decomposition.  Include every non-zero component
\(x\to\tau_{d+1}^{i}y\) in a single morphism whose target is the direct sum
of all the translates \(\tau_{d+1}^{i}y\) which occur.  Only finitely many
components occur because \(\mathcal C_d(H)\) is Hom-finite.  Taking the
direct sum over the chosen \(G\)-orbit representatives gives a finite
morphism
\[
f^\sharp:X^\sharp\longrightarrow B^\sharp
\]
in \(\Db[H]\).  After applying \(\pi\), the original minimal left
approximation factors through \(\pi f^\sharp\), by folding together the
copies of each indecomposable summand of \(B\).  Hence \(\pi f^\sharp\) is a
left \(\add(W)\)-approximation, although it need not be left minimal.  Lemma
\ref{lem-covering-approx} applied to \(\pi\) now shows that \(f^\sharp\) is a
left approximation with respect to the full subcategory generated by
\(\pi^{-1}W\).

Complete \(f^\sharp\) to a triangle
\[
X^\sharp\xrightarrow{f^\sharp} B^\sharp\longrightarrow
Y^\sharp\longrightarrow X^\sharp[1]
\]
in \(\Db[H]\).  By Lemma \ref{lem-covering-triangle}, its image under \(\pi\)
is a triangle.  Any left approximation is the direct sum of a minimal left
approximation and a redundant summand.  It follows that \(\pi Y^\sharp\) is
the corresponding direct sum of copies of \(Y\), together with an object in
\(\add(W)\).  In particular,
\(\add(W\oplus\pi Y^\sharp)=\add(W\oplus Y)\).

Now apply the covering functor \(F\).  By Lemma \ref{lem-covering-triangle},
we obtain a triangle in \(\mathcal T\):
\[
F X^\sharp\longrightarrow F B^\sharp\longrightarrow
F Y^\sharp\longrightarrow F X^\sharp[1].
\tag{\(\heartsuit\)}
\]
Put
\(X_{\mathcal T}=F X^\sharp,\ B_{\mathcal T}=F B^\sharp,\)
\(Y_{\mathcal T}=F Y^\sharp\), and \(W_{\mathcal T}=FW^\sharp\).
For every indecomposable summand \(U\) of \(W_{\mathcal T}\), every object
in \(F^{-1}(U)\) belongs to \(\pi^{-1}W\).  Therefore Lemma
\ref{lem-covering-approx}, now applied to \(F\), shows directly that the
first morphism in \((\heartsuit)\) is a left
\(\add(W_{\mathcal T})\)-approximation.  Since
\(F(W^\sharp\oplus X^\sharp)=F(\pi^{-1}M)\) is \(d\)-cluster tilting in
\(\mathcal T\), the rigidity condition
\(\Hom_{\mathcal T}(W_{\mathcal T},Y_{\mathcal T}[1])=0\)
implies, by applying \(\Hom_{\mathcal T}(W_{\mathcal T},-)\) to
\((\heartsuit)\), that the second morphism
\(B_{\mathcal T}\to Y_{\mathcal T}\) is a right
\(\add(W_{\mathcal T})\)-approximation.  Thus \((\heartsuit)\) is an
\(\add(W_{\mathcal T})\)-split triangle.

We regard the objects \(X_{\mathcal T},Y_{\mathcal T},W_{\mathcal T}\) and
\(B_{\mathcal T}\) as objects of the Frobenius model \(\mathcal E\).  Since
\(\mathcal T=\underline{\mathcal E}\), the triangle \((\heartsuit)\) is
induced by a conflation in \(\mathcal E\) after adding a projective direct
summand:
\[
0\longrightarrow X_{\mathcal T}
\longrightarrow B_{\mathcal T}\oplus P
\longrightarrow Y_{\mathcal T}
\longrightarrow 0,
\tag{\(\clubsuit\)}
\]
where \(P\in\add(\Lambda)\).  Adding the projective generator \(\Lambda\) to the
approximation subcategory turns \((\clubsuit)\) into an
\(\add(\Lambda\oplus W_{\mathcal T})\)-split sequence.  The
\(\mathcal D\)-split sequence result in \cite[Theorem 1.1]{Hu2011aa}
therefore gives a derived equivalence between
\(\End_{\mathcal E}(\Lambda\oplus W_{\mathcal T}\oplus X_{\mathcal T})\)
and
\(\End_{\mathcal E}(\Lambda\oplus W_{\mathcal T}\oplus Y_{\mathcal T})\).
Since the possible redundant summands of \(Y_{\mathcal T}\) belong to
\(\add(W_{\mathcal T})\), deleting repetitions gives a Morita equivalent
endomorphism algebra.  Thus this is the desired derived equivalence for one
mutation step after passing from \(\mathcal C_d(H)\) to \(\mathcal T\).
Composing these
equivalences along the \(G\)-mutation path gives the statement for
\(M_{\mathcal T}\) and \(M'_{\mathcal T}\).
\end{proof}

\begin{Coro}\label{coro-taud-power-derived-equivalence}
Assume the setting of Section \(3\), with \(H\) a finite dimensional
hereditary algebra, and suppose that
\(G=\langle \tau_{d+1}^m\rangle\)
for some \(m\geq1\).  Assume further that
\(\mathcal T=\underline{\mathcal E}\) for a Frobenius category \(\mathcal E\)
with projective generator \(\Lambda\).  Then \(\mathcal T\) admits
\(d\)-cluster tilting objects.  Moreover, if \(M,M'\in\T\) are \(d\)-cluster tilting objects, then the algebras $\End_{\mathcal E}(\Lambda\oplus M)$ and $\End_{\mathcal E}(\Lambda\oplus M')$ 
are derived equivalent.
\end{Coro}

\begin{proof}
The \(d\)-cluster category \(\mathcal C_d(H)\) has \(d\)-cluster tilting
objects, for example the canonical image of \(H\).  Since \(\tau_{d+1}\) acts
trivially on \(\mathcal C_d(H)\), the induced action of
\(G=\langle \tau_{d+1}^m\rangle\) on \(\mathcal C_d(H)\) is trivial.  Hence
every basic \(d\)-cluster tilting object in \(\mathcal C_d(H)\) is
\(G\)-stable.  Lemma \ref{lem-dct-H-ctcat} therefore gives
\(d\)-cluster tilting objects in \(\mathcal T\).

Let \(M\) and \(M'\) be as in the statement.  By Lemma
\ref{lem-dct-H-ctcat}, they lift to \(G\)-stable
\(d\)-cluster tilting objects in \(\mathcal C_d(H)\).  Since the induced
\(G\)-action is trivial, \(G\)-mutation reachability is ordinary mutation
reachability.  The \(d\)-cluster tilting objects in \(\mathcal C_d(H)\) are
mutation connected \cite{Zhou2009}, so Proposition
\ref{prop-G-mutation-derived-equivalence} gives the desired derived
equivalence.
\end{proof}

\begin{Def}\label{def-eventual-common-orbit}
Let \(\mathcal C\) be a \((d+1)\)-Calabi--Yau triangulated category, and let
\(G\) act on \(\mathcal C\) by triangle autoequivalences.  We say that the
pair \((\mathcal C,G)\) has the
\emph{eventual common-orbit property} if, for any two \(G\)-stable basic
\(d\)-cluster tilting objects \(M\) and \(N\) in \(\mathcal C\), there are
\(G\)-stable basic \(d\)-cluster tilting objects \(M'\) and \(N'\) such that
\(M\) and \(M'\) are \(G\)-mutation reachable, \(N\) and \(N'\) are
\(G\)-mutation reachable, and \(M'\) and \(N'\) have a common non-zero direct
summand. 
\end{Def}

\begin{Lem}\label{lem-shift-whole-mutation}
Let \(T\) be a basic \(G\)-stable \(d\)-cluster tilting object in
\(\mathcal C_d(k\Delta)\).  Then \(T[1]\) and \(T[-1]\) are
\(G\)-mutation reachable from \(T\).
\end{Lem}

\begin{proof}
Since every triangle autoequivalence in \(G\) commutes with the suspension,
\(T[1]\) and \(T[-1]\) are again \(G\)-stable.  Regard the whole object
\(T\) as the \(G\)-stable direct summand to be mutated, so that the
complement is zero.  The minimal right \(\add(0)\)-approximation of \(T\) is
the zero morphism \(0\to T\), and its exchange triangle is
\(T[-1]\longrightarrow 0\longrightarrow T\longrightarrow T.\)
Thus \(T[-1]\) is the \(G\)-right mutation of \(T\) at the whole summand.
Since \(G\)-mutation reachability is defined symmetrically, \(T[1]\) is also
reachable from \(T\).
\end{proof}

\begin{Prop}\label{prop-eventual-common-orbit-ADE}
In every Dynkin case allowed by Theorem \ref{theo-existence-dct}, the pair
\((\mathcal C_d(k\Delta),G)\) has the eventual common-orbit property.
\end{Prop}

\begin{proof}
By Lemma \ref{lem-shift-whole-mutation}, applying a power of the suspension to
the whole \(d\)-cluster tilting object is a \(G\)-mutation operation.  Thus it
is enough to find, in every \(G\)-stable \(d\)-cluster tilting object, a summand
belonging to a fixed suspension orbit.

In type \(\mathbb A_n\), every \((d+2)\)-angulation has an ear, whose remaining
edge is a shortest \(d\)-diagonal \(D_{i,i+d+1}\).  These shortest diagonals
form one \([1]\)-orbit.  In type \(\mathbb D_n\), every \((d+2)\)-angulation of
the punctured polygon has an ordinary ear, hence contains an ordinary short
\(d\)-arc \(D_{i,i+d+1}\); these ordinary short arcs also form one
\([1]\)-orbit.  Therefore two \(G\)-stable objects in types \(\mathbb A\) or
\(\mathbb D\) can be shifted until the chosen short arcs coincide.

For type $\mathbb{E}$ we use the AR-quiver notation from the proof of Theorem
\ref{theo-existence-dct}.  Since \(\tau\simeq[d]\) in
\(\mathcal C_d(k\mathbb E_n)\), powers of \(\tau\) are again reachable by
whole-object \(G\)-mutation.  The finite Hom-hammock and orbit-graph
calculations in Theorem \ref{theo-counting-dct} give a fixed level which every
\(G\)-stable object must meet: for \(\mathbb E_6\), and for \(\mathbb E_7\) and
\(\mathbb E_8\) with \(L_G=1\), every object meets each horizontal
\(\tau\)-orbit; for \(\mathbb E_7\) with \(L_G=7\), every object is a single
orbit on level \(7\); for \(\mathbb E_8\) with \(L_G=2\), every non-zero level
multiset in Table \ref{tab:E8-L2-clique-counts} contains level \(8\); and for
\(\mathbb E_8\) with \(L_G=4\), every edge counted by Table
\ref{tab:E8-L4-Dij} contains a level \(8\) orbit.  Choosing such a level in two
objects and applying a suitable power of \(\tau\), the two chosen \(G\)-orbits
become equal.  This gives a common non-zero  direct summand.
\end{proof}

We recall the form of Iyama--Yoshino reduction used below.  Let
\(\mathcal C\) be a \((d+1)\)-Calabi--Yau triangulated category, and let
\(\mathcal D\) be a functorially finite \(d\)-orthogonal subcategory, that is,
\(\Hom_{\mathcal C}(\mathcal D,\mathcal D[i])=0\) for \(1\le i\le d\).  Set
\[
\mathcal Z_{\mathcal D}
=
\{Y\in\mathcal C\mid
\Hom_{\mathcal C}(\mathcal D,Y[i])=0,\ 1\le i\le d\}.
\]
The Iyama--Yoshino reduction is
\(\mathcal C_{\mathcal D}:=\mathcal Z_{\mathcal D}/[\mathcal D]\), where
\([\mathcal D]\) is the ideal of morphisms factorizing through objects of
\(\mathcal D\).  This quotient is again triangulated; we denote its
suspension by \(\langle1\rangle\).  The reduced suspension is constructed
from the original suspension by approximation triangles: for
\(Z\in\mathcal Z_{\mathcal D}\), choose a left \(\mathcal D\)-approximation
\(Z\to D_Z\) and complete it to a triangle
\[
Z\longrightarrow D_Z\longrightarrow Z\langle1\rangle\longrightarrow Z[1].
\]
Dually, \(\langle-1\rangle\) is obtained from right
\(\mathcal D\)-approximations.  Thus \(\langle1\rangle\) is not simply the
restriction of \([1]\), but it agrees with the image of \([1]\) whenever the
corresponding approximation term vanishes; for example, if the successive
right \(\mathcal D\)-approximations of \(Z,Z[-1],\ldots,Z[-i+1]\) are zero,
then \(\overline Z\langle-i\rangle\simeq\overline{Z[-i]}\) in
\(\mathcal C_{\mathcal D}\).  Since both \(\mathcal C\) and
\(\mathcal C_{\mathcal D}\) are \((d+1)\)-Calabi--Yau in this setting, their
AR-translations satisfy
\(\tau_{\mathcal C}\simeq[d]\) and
\(\tau_{\mathcal C_{\mathcal D}}\simeq\langle d\rangle\).  Consequently the
AR-translation in the reduction is obtained from the original \([d]\) by the
same reduction procedure.

We shall use the following standard properties from \cite{Iyama2008ac}: if
\(T\) is a \(d\)-cluster tilting object of \(\mathcal C\) containing
\(\mathcal D\) as a direct summand, then the image of the complement of
\(\mathcal D\) is a \(d\)-cluster tilting object in
\(\mathcal C_{\mathcal D}\), and its endomorphism algebra is obtained by
factoring out morphisms through \(\mathcal D\).  We also use the transitivity
of reductions: reducing successively by \(d\)-orthogonal summands gives the
same category as reducing by their direct sum.

We now record the reduction-closedness statement needed for the induction on
common \(G\)-stable summands.

\begin{Prop}\label{prop-dynkin-reduction-closed}
Let \(d\geq1\), let \(\Delta\) be a Dynkin diagram of type \(ADE\), and put
\(\mathcal C=\mathcal C_d(k\Delta)\).  Let \(\mathcal D\) be a functorially
finite \(d\)-orthogonal subcategory of \(\mathcal C\).  Then there exist
Dynkin diagrams \(\Delta_1,\ldots,\Delta_r\) such that
\(\mathcal C_{\mathcal D}\simeq \prod_{i=1}^r \mathcal C_d(k\Delta_i)\)
as triangulated categories.
\end{Prop}

\begin{proof}
We first prove the statement when \(\mathcal D=\add(X)\) with \(X\)
indecomposable.  Choose a lift of \(X\) to \(\Db[k\Delta]\).  Since \(\Delta\)
is Dynkin, this lift belongs to a complete slice.  Equivalently, after replacing
\(k\Delta\) by a derived equivalent hereditary algebra \(kQ\) with the same
underlying Dynkin diagram, we may assume that
\(X=P_x\)
is an indecomposable projective summand of the standard \(d\)-cluster tilting
object
\(M=\pi(kQ)=\bigoplus_{v\in Q_0}P_v\)
in \(\mathcal C_d(kQ)\).  We choose the orientation of \(Q\) so that \(x\) is a
sink, and hence
\(\Hom_{\mathcal C}(P_x,T)=0,\ T:=\bigoplus_{v\ne x}P_v .\)
Thus \(M=P_x\oplus T\) is a \(d\)-cluster tilting object of \(\mathcal C\)
(an ordinary cluster tilting object when \(d=1\)).

Let \(\mathcal C_X=\mathcal Z_X/[X]\) be the reduction, and denote its
suspension functor by \(\langle1\rangle\).  By the basic property of
Iyama--Yoshino reduction \cite{Iyama2008ac}, the image \(\overline T\) of
\(T\) is a \(d\)-cluster tilting object of \(\mathcal C_X\).  Moreover,
\(\End_{\mathcal C_X}(\overline T)\simeq \End_{\mathcal C}(T)/[X](T,T).\)
Since \(\Hom_{\mathcal C}(X,T)=0\), no non-zero endomorphism of \(T\) factors
through \(X\).  Therefore
\(\End_{\mathcal C_X}(\overline T)\simeq \End_{\mathcal C}(T)\simeq
k(Q\setminus\{x\})\)
up to opposite algebra.  The full subquiver \(Q\setminus\{x\}\) is a disjoint
union of Dynkin quivers,
\(Q\setminus\{x\}=Q_1\sqcup\cdots\sqcup Q_r.\)
Consequently \(\End_{\mathcal C_X}(\overline T)\simeq\prod_{i=1}^r kQ_i\) is
hereditary of Dynkin type.

It remains to check the negative extension condition in the characterization
of higher cluster categories in \cite[Theorem 4.2]{Keller2008ab}.  We apply
that theorem with its Calabi--Yau parameter equal to \(d+1\); with this
convention, its cluster tilting object is precisely a \(d\)-cluster tilting
object in our notation.  Since \(M=\pi(kQ)\) is the canonical image of the
projective \(kQ\)-module \(kQ\), its negative self-extensions vanish in the
higher cluster category; this is the computation in
\cite[Section 4.1]{Keller2008ab}, with parameter \(d+1\).  Thus
\(\Hom_{\mathcal C}(M,M[-i])=0,\ 1\le i\le d-1.\)
Since \(T\) is a direct summand of \(M\), we get
\(\Hom_{\mathcal C}(T,T[-i])=0,\ 1\le i\le d-1.\)
On the other hand, the explicit description of the suspension in the reduction
gives
\(\overline T\langle -i\rangle\simeq \overline{T[-i]},\ 1\le i\le d-1,\)
because \(\Hom_{\mathcal C}(X,T)=0\) and
\(\Hom_{\mathcal C}(X,T[-j])=0\) for \(1\le j\le d-2\).  Hence
\(\Hom_{\mathcal C_X}(\overline T,\overline T\langle -i\rangle)\)
is a quotient of \(\Hom_{\mathcal C}(T,T[-i])\), and therefore vanishes for
\(1\le i\le d-1\).  Thus \(\mathcal C_X\) is Hom-finite, algebraic and
\((d+1)\)-Calabi--Yau, and it contains a \(d\)-cluster tilting object whose
endomorphism algebra is hereditary and satisfies the required negative
extension vanishing.  The characterization theorem therefore yields
\[
\mathcal C_X
\simeq
\mathcal C_d\!\left(\End_{\mathcal C_X}(\overline T)\right)
\simeq
\mathcal C_d\!\left(\prod_{i=1}^r kQ_i\right)
\simeq
\prod_{i=1}^r \mathcal C_d(kQ_i).
\]
Writing \(\Delta_i\) for the underlying Dynkin diagram of \(Q_i\), the proposition
follows for \(\mathcal D=\add(X)\).

Now let
\(\mathcal D=\add(X_1\oplus\cdots\oplus X_s)\)
be arbitrary.  Since \(\mathcal D\) is \(d\)-orthogonal, after reducing by
\(X_1,\ldots,X_{j-1}\), the image of \(X_j\) is still \(d\)-orthogonal to the
images of the remaining summands.  The transitivity of Iyama--Yoshino
reduction gives
\[
\mathcal C_{\mathcal D}
\simeq
(((\mathcal C_{X_1})_{\overline X_2})\cdots)_{\overline X_s}.
\]
The one-summand case applies at each step; the relevant \(d\)-orthogonality
condition is inherited after reduction, and reductions of finite products are
taken componentwise.  Hence the resulting category remains a finite product of
\(d\)-cluster categories of Dynkin type.
\end{proof}

\begin{Prop}\label{prop-dynkin-G-mutation-connected}
Let \(\Delta\) be a Dynkin diagram of type \(ADE\), let \(d\geq1\), and let
\(\mathcal C=\mathcal C_d(k\Delta)\).  Let \(G=\langle\sigma\rangle\) be a
cyclic group acting on \(\mathcal C\) by triangle autoequivalences, and assume
that the induced action on the AR-quiver is weakly admissible.  Then the
\(G\)-stable basic \(d\)-cluster tilting objects in \(\mathcal C\) are
\(G\)-mutation connected.
\end{Prop}

\begin{proof}
We prove a slightly stronger statement, allowing the category to be a finite
product \(\mathcal C=\prod_{\alpha\in A}\mathcal C_d(k\Delta_\alpha)\) of
Dynkin \(d\)-cluster categories which occurs by successive Iyama--Yoshino
reductions.  The group \(G=\langle\sigma\rangle\) is allowed to permute the
factors, and we assume only that \(\mathcal C/G\) is again a locally finite
triangulated category, equivalently, that the induced action on each connected
AR-component is weakly admissible.

This hypothesis is stable under the reductions used below.  If \(\mathcal D\)
is a \(G\)-stable \(d\)-orthogonal summand, then \(G\) acts on the reduction
\(\mathcal C_{\mathcal D}\).  The orbit category
\(\mathcal C_{\mathcal D}/G\) is the corresponding reduction of
\(\mathcal C/G\), and is again a triangulated category with finitely many
indecomposable objects.  Hence, after decomposing \(\mathcal C_{\mathcal D}\)
as a finite product of Dynkin \(d\)-cluster categories by Proposition
\ref{prop-dynkin-reduction-closed}, the stabilizer of each factor cycle acts
weakly admissibly on a representative factor. 

We argue by induction on the number of indecomposable summands of a basic
\(d\)-cluster tilting object in \(\mathcal C\).  First we show that
\((\mathcal C,G)\) has the eventual common-orbit property.  For a single Dynkin
factor, the existence of \(G\)-stable \(d\)-cluster tilting objects forces the
weakly admissible action to satisfy the conditions of Theorem
\ref{theo-existence-dct}; hence Proposition
\ref{prop-eventual-common-orbit-ADE} applies.  If \(G\) permutes several
factors, decompose the factors into \(\sigma\)-cycles.  For a cycle of length
\(m\), choose one representative factor \(\mathcal C_{\alpha_0}\).  A
\(G\)-stable \(d\)-cluster tilting object on this cycle is determined by its
component on \(\mathcal C_{\alpha_0}\), and that component is stable under the
stabilizer \(\langle\sigma^m\rangle\).  As explained above, this stabilizer
acts weakly admissibly; since such stable objects exist, Theorem
\ref{theo-existence-dct} puts it among the allowed Dynkin cases.  Applying
Proposition \ref{prop-eventual-common-orbit-ADE} on the representative factor
and transporting the resulting mutations around the cycle gives the eventual
common-orbit property for the whole cycle.  Taking the product over all cycles
gives the property for \((\mathcal C,G)\).

Let \(M\) and \(N\) be two \(G\)-stable basic \(d\)-cluster tilting objects.
Using the eventual common-orbit property, first \(G\)-mutate them to
\(M'\) and \(N'\) which have a common non-zero \(G\)-stable direct summand
\(D\).  It remains to connect \(M'\) and \(N'\).  Write
\(M'=D\oplus M_0\) and \(N'=D\oplus N_0\), and reduce at \(D\).  By Proposition
\ref{prop-dynkin-reduction-closed}, the reduction \(\mathcal C_D\) is again a
finite product of Dynkin \(d\)-cluster categories.  By the stability of
weakly admissible actions under \(G\)-stable reductions discussed above, the
induced action of \(G\) on \(\mathcal C_D\) satisfies the same induction
hypothesis.  Therefore the images of \(M_0\) and \(N_0\) in \(\mathcal C_D\)
are \(G\)-mutation connected.

It remains to explain why a \(G\)-mutation step in the reduction lifts to a
\(G\)-mutation in \(\mathcal C\) which keeps \(D\) fixed.  Consider one right
\(G\)-mutation step in \(\mathcal C_D\):
\(\overline U=\overline X\oplus \overline L\) is changed to
\(\overline U^*=\overline Y\oplus \overline L\), with exchange triangle
\(\overline Y\to \overline B\to \overline X\to
\overline Y\langle1\rangle\), where \(\overline B\to\overline X\) is a
minimal right \(\add(\overline L)\)-approximation.  Choose representatives
\(X,L,B,Y\in\mathcal Z_D\) without direct summands in \(\add(D)\); since \(D\)
is \(G\)-stable, the representatives of the \(G\)-stable objects
\(\overline X,\overline L,\overline Y\) may be chosen \(G\)-stable.  Lift the
map \(\overline B\to\overline X\) to a morphism \(b:B\to X\) in
\(\mathcal C\), and let \(d_X:D_X\to X\) be a right \(\add(D)\)-approximation.
Then
\((b,d_X):B\oplus D_X\to X\)
is a right \(\add(L\oplus D)\)-approximation: maps from \(L\) factor through
\(b\) modulo morphisms factoring through \(\add(D)\), and the remaining part,
as well as every map from \(\add(D)\), factors through \(d_X\).  After deleting
redundant direct summands, it is minimal.

Complete this minimal approximation to a triangle
\(Y_0\to B_0\to X\to Y_0[1]\) in \(\mathcal C\), with
\(B_0\in\add(L\oplus D)\).  By the construction of triangles in the
Iyama--Yoshino reduction, its image in
\(\mathcal C_D=\mathcal Z_D/[\mathcal D]\) is the given exchange triangle in
the reduction, so \(\overline{Y_0}\simeq\overline Y\).
Removing possible direct summands of \(Y_0\) lying in \(\add(D)\), we may take
the lifted exchange complement to be \(Y\).  Therefore
\(D\oplus L\oplus Y\) is obtained from \(D\oplus L\oplus X\) by a
\(G\)-right mutation at the \(G\)-stable summand \(X\), with complement
\(D\oplus L\).  The same argument, with the arrows reversed, treats a
left mutation step.  Hence the entire \(G\)-mutation path in
\(\mathcal C_D\) lifts to a \(G\)-mutation path in \(\mathcal C\) fixing
\(D\).  Thus \(M'\) and \(N'\), and hence also \(M\) and \(N\), are
\(G\)-mutation reachable.
\end{proof}

\begin{Theo}\label{theo-equivariant-mutation-connected}
Let \(\mathcal T\) be a Hom-finite Krull--Schmidt triangulated category
with only finitely many indecomposable objects.  Assume that \(\mathcal T\)
admits \(d\)-cluster tilting objects and that
\(\mathcal T=\underline{\mathcal E}\) for a Frobenius category \(\mathcal E\)
with projective generator \(\Lambda\).  Then the endomorphism algebras of any
two \(d\)-cluster tilting objects in \(\mathcal E\) are derived equivalent.
\end{Theo}

\begin{proof}
Replacing a \(d\)-cluster tilting object by its basic additive generator only
changes its endomorphism algebra by Morita equivalence, so we may assume the
objects are basic.  Let \(T\) and \(T'\) be two basic \(d\)-cluster tilting
objects in \(\mathcal E\).  Since projective objects are Ext-orthogonal to all
objects, the maximality condition for \(d\)-cluster tilting objects gives
\(\Lambda\in\add(T)\cap\add(T')\).  Thus we may write
\(T=\Lambda\oplus M_{\mathcal T}\) and
\(T'=\Lambda\oplus M'_{\mathcal T}\), where \(M_{\mathcal T}\) and
\(M'_{\mathcal T}\) denote their images in the stable category.  The standard
correspondence between \(d\)-cluster tilting objects in a Frobenius category
containing the projectives and those in its stable category shows that
\(M_{\mathcal T}\) and \(M'_{\mathcal T}\) are \(d\)-cluster tilting objects in
\(\mathcal T=\underline{\mathcal E}\).

Since \(\mathcal T\) has only finitely many indecomposable objects, it is
locally finite.  We first reduce to the connected components of its AR-quiver.
Write \(\mathcal T=\prod_{\lambda\in I}\mathcal T_\lambda\) as the finite
product of its connected components.  Then
\(M_{\mathcal T}=\bigoplus_{\lambda}M_\lambda\) and
\(M'_{\mathcal T}=\bigoplus_{\lambda}M'_\lambda\), and the \(d\)-cluster
tilting condition is checked componentwise.

Fix one component \(\mathcal T_\lambda\).  By the classification of locally
finite triangulated categories, its AR-quiver is of the form
\(\mathbb Z\Delta_\lambda/G_\lambda\), where \(\Delta_\lambda\) is Dynkin and
\(G_\lambda\) is one of the weakly admissible cyclic groups in Theorem
\ref{thm-locally-finite-classification}.  Hence this component is in the
covering situation of Section \(3\).  Via the canonical quotient
\(\Db[k\Delta_\lambda]\to\mathcal C_d(k\Delta_\lambda)\), Lemma
\ref{lem-dct-H-ctcat} identifies \(M_\lambda\) and \(M'_\lambda\) with
\(G_\lambda\)-stable \(d\)-cluster tilting objects in
\(\mathcal C_d(k\Delta_\lambda)\).  Proposition
\ref{prop-dynkin-G-mutation-connected} then gives a \(G_\lambda\)-mutation path
between them.

Taking the product over all connected components gives a covering for the
finite product of the Dynkin \(d\)-cluster categories, and the componentwise
paths concatenate to a \(G\)-mutation path from the lift of \(M_{\mathcal T}\)
to the lift of \(M'_{\mathcal T}\).  Proposition
\ref{prop-G-mutation-derived-equivalence} therefore gives a derived
equivalence between
\(\End_{\mathcal E}(\Lambda\oplus M_{\mathcal T})\) and
\(\End_{\mathcal E}(\Lambda\oplus M'_{\mathcal T})\).  These are precisely
\(\End_{\mathcal E}(T)\) and \(\End_{\mathcal E}(T')\).
\end{proof}

\begin{Coro}\label{coro-finite-frobenius-derived-equivalence}
Let \(R\) be a Cohen--Macaulay finite Iwanaga--Gorenstein algebra, and let
\(d\geq1\).  If \(\operatorname{CM} R\) admits \(d\)-cluster tilting objects,
then the endomorphism algebras of any two \(d\)-cluster tilting objects in
\(\operatorname{CM} R\) are derived equivalent.  In particular, the same
conclusion holds for the module category \(\modcat{A}\) of a
representation-finite self-injective algebra \(A\), whenever \(\modcat{A}\)
admits \(d\)-cluster tilting objects.
\end{Coro}

\begin{proof}
The category \(\operatorname{CM} R\) is Frobenius, with projective generator
\(R\), and its stable category \(\underline{\operatorname{CM}}R\) is
triangulated.  Since \(R\) is Cohen--Macaulay finite, this stable category has
only finitely many indecomposable objects.  Hence Theorem
\ref{theo-equivariant-mutation-connected} applies and gives the desired
derived equivalence.  If \(A\) is self-injective, then
\(\modcat{A}=\operatorname{CM} A\); if \(A\) is representation-finite, then it
is Cohen--Macaulay finite.  Hence the self-injective case is a special case.
\end{proof}

\begin{Coro}\label{coro-CMfinite-isolated-NCCR}
Let \(R\) be a commutative Cohen--Macaulay finite normal Gorenstein isolated
singularity of Krull dimension \(n\geq3\).  Suppose that \(M,N\in
\operatorname{CM}R\) are \((n-2)\)-cluster tilting generators and that
\(
\Lambda=\End_R(M), \Gamma=\End_R(N)
\)
are noncommutative crepant resolutions of \(R\).  Then \(\Lambda\) and
\(\Gamma\) are derived equivalent.
\end{Coro}

\begin{proof}
Since \(R\) is Gorenstein, the category \(\operatorname{CM}R\) is Frobenius
with projective generator \(R\).  By the Cohen--Macaulay finiteness
assumption, Corollary \ref{coro-finite-frobenius-derived-equivalence} applies
with \(d=n-2\) to the two cluster tilting objects \(M\) and \(N\).  Hence
\(\End_R(M)\) and \(\End_R(N)\) are derived equivalent.  
\end{proof}

\section{Applications to rigidity dimensions}
\label{subsec-rigidity-dimension-self-injective}

We now explain how the existence criterion for \(d\)-cluster tilting objects
can be used to compute rigidity dimensions for some representation-finite
self-injective algebras.  Recall first the definition introduced in
\cite{ChenFangKernerKoenigYamagata2021}.  For a finite-dimensional algebra
\(A\), its rigidity dimension is
\[
\rigdim A=\sup\left\{\domdim \End_A(M)\ \middle|\
\begin{array}{l}
M\text{ is a generator-cogenerator,}\\
\gldim\End_A(M)<\infty
\end{array}\right\}.
\]
Thus rigidity dimension is a dominant-dimension analogue of representation
dimension.  It is designed to measure how good the best finite global
dimension resolutions of \(A\) can be.  One useful consequence is the
following Hochschild-cohomological restriction: if \(\rigdim A=q\), then
\(HH^*(A)\) can have non-nilpotent homogeneous generators only in degree zero
and in degrees strictly larger than \(q-2\), see
\cite{ChenFangKernerKoenigYamagata2021}.  Its finiteness is not known in
complete generality; for representation-finite self-injective algebras,
however, the problem is reduced to a finite Auslander--Reiten calculation.

For self-injective algebras the link with rigidity degrees is especially
simple.  If \(M\) is a generator-cogenerator, let \(\rd(M)\) be the largest
integer \(r\) such that \(\Ext_A^i(M,M)=0\) for \(1\leq i\leq r\).  M\"uller's
criterion gives \(\domdim\End_A(M)=\rd(M)+2\) \cite{Muller1968ab}.  Hence,
for a non-semisimple representation-finite self-injective algebra \(A\), put
\[
d_{\max}=\max\{\rd(X)\mid X\text{ is an indecomposable non-projective }
A\text{-module}\}.
\]
Every generator-cogenerator \(M\) with \(\gldim\End_A(M)<\infty\) has a
non-projective summand, and \(\rd(M)\leq d_{\max}\).  Therefore
\(\rigdim A\leq d_{\max}+2\).  If, on the other hand, the stable category
\(\stmodcat{A}\) has a \(d_{\max}\)-cluster tilting object \(T\), then
\(A\oplus T\) is a maximal \(d_{\max}\)-orthogonal generator-cogenerator.
By Iyama's higher Auslander correspondence \cite{Iyama2007ab},
\(\End_A(A\oplus T)\) has finite global dimension, and M\"uller's criterion
then gives \(\rigdim A=d_{\max}+2\).  Thus the calculation has two steps:
first find \(d_{\max}\) from the rigidity-degree formulas of
\cite{HuYinRigidityDegrees}, and then use Theorem \ref{theo-existence-dct} to
check whether a \(d_{\max}\)-cluster tilting object exists.

We use the harmless convention that a \(0\)-cluster tilting object in a finite
Krull--Schmidt category is the additive generator given by the direct
sum of all indecomposable objects.  Hence, if \(d_{\max}=0\), then a
\(d_{\max}\)-cluster tilting object always exists and the preceding paragraph
gives \(\rigdim A=2\).  In the following type \(\mathbb A\) families we
therefore assume \(d_{\max}>0\).

\begin{Coro} 
\label{coro-rigdim-type-A-closed-families}
Let \(A\) be an indecomposable non-semisimple representation-finite
self-injective algebra of tree class \(\mathbb A\), and let \(\Gamma_A\) be
the stable AR-quiver of \(A\).  Assume \(d_{\max}>0\).  Then \(A\) has
\(d_{\max}\)-cluster tilting modules if and only if one of the following
conditions holds.  In each case  \(\rigdim A=d_{\max}+2\).

\begin{enumerate}[label={\rm(A\arabic*)}]
\item \(\Gamma_A=\mathbb Z\mathbb A_{m-1}/\langle\tau^R\rangle\), and one of
the following conditions holds.
\begin{enumerate}[label={\rm(\alph*)}]
\item \(m=2\).  Then \(d_{\max}=R-1\) and \(\rigdim A=R+1\).
\item \(m\geq3\) and \(R=am-s\), where \(a\geq1\), \(s\mid m-1\), and
\(1\leq s\leq m-1\).  Then
\(d_{\max}=2a(m-1)/s-2\) and \(\rigdim A=2a(m-1)/s\).
\item \(m=2p\geq4\) and \(R=c(m-2)\), where \(c\mid p\).  Then
\(d_{\max}=m-3\) and \(\rigdim A=m-1\).
\end{enumerate}
\item \(\Gamma_A=\mathbb Z\mathbb A_{m-1}/
\langle\tau^{u(m-1)}\phi\rangle\), where \(m\geq4\) is even and
\(\phi=\tau^{m/2}S\).  Put \(R=u(m-1)-m/2\).  If \(s\) is an odd integer with
\(1\leq s\leq m-1\), \(R\equiv -s\pmod m\), and \(s\mid m-1\), then
\(d_{\max}=(2u(m-1)^2-2s)/(sm)\) and
\(\rigdim A=(2u(m-1)^2-2s)/(sm)+2\).
\end{enumerate}
\end{Coro}

\begin{proof}
We first recall the input from \cite{HuYinRigidityDegrees}.  With the
labelling used there, the boundary indecomposable modules in the stable
AR-quiver have rigidity degree denoted by \(\rd(1)\).  The monotonicity
statements \cite[Proposition 4.3(3), Corollary 4.7]{HuYinRigidityDegrees}
show that, in the families considered here, \(d_{\max}=\rd(1)\).  For
\(\Gamma_A=\mathbb Z\mathbb A_{m-1}/\langle\tau^R\rangle\), the value
\(\rd(1)\) is determined by the weighted Fibonacci sequence attached to the
Euclidean algorithm for the pair \((m,R)\).  For
\(\Gamma_A=\mathbb Z\mathbb A_{m-1}/
\langle\tau^{u(m-1)}\phi\rangle\), with \(m\) even and
\(\phi=\tau^{m/2}S\), put \(R=u(m-1)-m/2\); then the relevant pair is
\(
M=R+m,Q=2R+m\).
These are precisely the two type \(\mathbb A\) formulas in
\cite[Theorem 4.1]{HuYinRigidityDegrees}. For each case, let $\mathbf{k}$ be the sequence of coefficients in the Euclidean algorithm for the pair, and let \(F_i, -1\leq i\leq |\mathbf{k}|\) be the weighted Fibonacci numbers defined by the recurrence relation \(F_{-1}=0, F_0=1, F_{i+2}=k_{i+2}F_{i+1}+F_i\). We refer to \cite{HuYinRigidityDegrees} for the details of the Euclidean algorithm and the weighted Fibonacci numbers.

We first treat \(\Gamma_A=\mathbb Z\mathbb A_{m-1}/\langle\tau^R\rangle\).  If $(m,R)=1$,  then \cite[Theorem 4.1]{HuYinRigidityDegrees} tells us that \(\rd(1)=2F_{|\mathbf{k}|-1}\) when $|\mathbf{k}|$ is even and \(\rd(1)=2(F_{|\mathbf{k}|}-F_{|\mathbf{k}|-1})\) when $|\mathbf{k}|$ is odd. In both cases, we have  \(\frac{\rd(1)}{2}m\equiv -1\pmod{R}\) by  \cite[Proposition 3.4]{HuYinRigidityDegrees}. If \((m,R)>1\), then \(\rd(1)=2F_{|\mathbf{k}|}-1\), where $F_{|\mathbf{k}|}=R/(m,R)$ and thus \(\frac{\rd(1)+1}{2}m=F_{|\mathbf{k}|}m\equiv 0\pmod{R}\).

For \(m=2\), the formula in \cite[Theorem 4.1]{HuYinRigidityDegrees} gives
\(\rd(1)=R-1\). For $d_{\max}=R-1$, $N=(R-1)m+2=2R$, and 
\[
L_G=\frac{2R}{(2R,R(R-1))}=
\begin{cases}
1,&R\text{ odd},\\
2,&R\text{ even}.
\end{cases}
\]
Thus both cases satisfies Theorem \ref{theo-existence-dct}.  This proves
(A1)(a).

Now let \(m\geq 3\), and assume first that \((m,R)=1\). By the discussion above, we can assume that 
\(\frac{\rd(1)}{2}m+1=\lambda R.
\) for some positive integer \(\lambda\). In this case $N=m\rd(1)+2=2\lambda R$. Write \(c=\rd(1)/2\). Then $(c,\lambda)=1$ and 
\[
L_G=\frac{2\lambda R}{(2\lambda R,2cR)}=\lambda .
\]
The case \(L_G=2\) with \(m-1\) odd in Theorem
\ref{theo-existence-dct} cannot occur here: it would force
\(\lambda=2\), while then \(m\) is even and \(cm+1=2R\) has odd left hand
side and even right hand side. Hence Theorem \ref{theo-existence-dct} says that a \(\rd(1)\)-cluster tilting module exists if and only if   \(L_G=\lambda\mid (m-1,\rd(1)+2)\), that is, \(\lambda\mid (m-1,2c+2)\). In this case, put \(s=(m-1)/\lambda\).  Since
\(m\equiv1\pmod\lambda\), the equality \(cm+1=\lambda R\) gives
\(c+1\equiv0\pmod\lambda\). Write \(c+1=a\lambda\). It follows that $\lambda\mid 2c+2$ automatically. Now
\(
\lambda R=cm+1=(a\lambda-1)m+1=\lambda(am-s)
\),
so \(R=am-s\), \(a\geq 1\), \(1\leq s\leq m-1\) and \(s\mid m-1\). In this case, 
\(
d_{\max}=\rd(1)=2c=2a\lambda-2=2a(m-1)/s-2\).  This proves (A1)(b).

It remains in the pure \(\tau^R\)-case to consider \((m,R)>1\). Put
\(g=(m,R)\), and write \(m=gE\), \(R=gF\), with \((E,F)=1\). In this case,
we have already seen that \(d_{\max}=\rd(1)=2F-1\), where \(F=R/(m,R)\).
Set \(D=(md_{\max}+2,Rd_{\max})\). Since \(Fm=ER\), we have
\[
D\mid F(md_{\max}+2)-E(Rd_{\max})=2F.
\]
Thus \(D\leq 2F\) and \(L_G=(md_{\max}+2)/D\). For \(m\geq3\), the
exceptional type \(\mathbb A\) alternative \(L_G=2\) with \(m-1\) odd cannot
occur: it would give \(D=(md_{\max}+2)/2\leq2F\), hence
\(m(2F-1)+2\leq4F\), which forces \(m\leq2\). Therefore Theorem
\ref{theo-existence-dct} forces \(L_G\mid m-1\) and
\(L_G\mid d_{\max}+2=2F+1\). Since \(md_{\max}+2=DL_G\) and \(D\leq2F\),
these two divisibilities give
\[
m(2F-1)+2=DL_G\leq 
2F(m-1),\qquad
m(2F-1)+2=DL_G\leq 2F(2F+1).
\]
The first inequality gives \(m\geq2F+2\), while the second gives
\(m\leq2F+2\). Hence \(m=2F+2\). Thus \(m=2p\) and \(F=p-1\). Moreover
equality must hold in the two estimates, so \(D=2F\) and
\(L_G=2F+1=m-1=d_{\max}+2\). Now
\[
D=(2F(2F+1),gF(2F-1))=F(2(2F+1),g(2F-1)).
\]
Since \(2F-1\) is coprime to \(2(2F+1)\), the equality \(D=2F\) is
equivalent to \((2(2F+1),g)=2\). Because \(g\mid m=2(F+1)\), this means
\(g=2c\) with \(c\mid F+1=p\). Therefore
\(R=gF=2cF=c(m-2)\), and \(d_{\max}=2F-1=m-3\). Conversely, if \(m=2p\) and \(R=c(m-2)\) with
\(c\mid p\), then
\[
N=md_{\max}+2=(m-1)(m-2),\qquad Rd_{\max}=c(m-2)(m-3).
\]
Because \(c\mid m/2\), \((m-1,c)=1\), and \((m-1,m-3)=1\), we get
\((N,Rd_{\max})=m-2\).  Thus \(L_G=N/(N,Rd_{\max})=m-1=d_{\max}+2\), and Theorem
\ref{theo-existence-dct} applies.  This proves (A1)(c).

We now treat \(\Gamma_A=\mathbb Z\mathbb A_{m-1}/
\langle\tau^{u(m-1)}\phi\rangle\).  In the notation of \cite[Theorem 4.1]{HuYinRigidityDegrees} for type
\((A_{m-1},u,2)\), the integer denoted there by \(n\) is our \(R\), and the
two integers used in the Euclidean algorithm are \(M=m+n=R+m\) and
\(N=m+2n=2R+m=Q\).  The boundary vertex \(t=1\) has rigidity degree
\(d_{\max}=\rd(1)\).  
For any value of \(d\), the degree of the generator
\(\tau^{u(m-1)}\phi\), with \(\phi=\tau^{m/2}S\), is \(Md+1\).  Hence the induced order is
\[
L_G=\frac{md+2}{(md+2,Md+1)}.
\]

We first exclude the case \((M,Q)>1\).  Put \(g=(M,Q)\), and write
\(M=gE\), \(Q=gF\).  Since \(Q=2M-m\), we have \(m=g(2E-F)\).  By
\cite[Theorem 4.1]{HuYinRigidityDegrees}, the boundary rigidity degree is
\(d_{\max}=\rd(1)=F-1\) in this non-coprime case.  Set
\(D=(md_{\max}+2,Md_{\max}+1)\).  Then
\[
D\mid E(md_{\max}+2)-(2E-F)(Md_{\max}+1)=F.
\]
Hence \(D\leq F\).  Moreover \(g\mid m\), and from
\(Q=2u(m-1)\) we get \(g\mid2u\), because \((g,m-1)=1\).  Thus
\(
F=Q/g=(2u/g)(m-1)\geq m-1\). 
Consequently
\[
L_G=\frac{md_{\max}+2}{D}\geq\frac{m(F-1)+2}{F}>m-1.
\]
The type \(\mathbb A\) condition in Theorem \ref{theo-existence-dct} is
therefore impossible: \(L_G\) cannot divide \(m-1\), and the exceptional
alternative \(L_G=2\) is also impossible.  Thus the existence of a
\(d_{\max}\)-cluster tilting module forces \((M,Q)=1\).

Now assume \((M,Q)=1\).  Writing \(|\mathbf{k}|\) for the length of the coeffecient sequence occuring in the  Euclidean
algorithm as in \cite{HuYinRigidityDegrees} and \(\mathsf F_i\) for the weighted Fibonacci numbers, the last positive remainder is
\(1\).  If \(|\mathbf{k}|\) is even, then
\cite[Theorem 4.1]{HuYinRigidityDegrees} gives
\(d_{\max}=\mathsf F_{|\mathbf{k}|-1}\), and
\cite[Proposition 3.4]{HuYinRigidityDegrees} gives
\(\operatorname{rem}(d_{\max}M)_Q=Q-1\).  If \(|\mathbf{k}|\) is odd, then
\cite[Theorem 4.1]{HuYinRigidityDegrees} gives
\(d_{\max}=\mathsf F_{|\mathbf{k}|}-\mathsf F_{|\mathbf{k}|-1}\), and the equality case in
\cite[Proposition 3.4]{HuYinRigidityDegrees} again gives
\(\operatorname{rem}(d_{\max}M)_Q=Q-1\).  Thus, in both coprime cases,
\(
d_{\max}M\equiv -1\pmod {Q}\). 
Write
\(
d_{\max}M+1=bQ\). 
Since \(2(d_{\max}M+1)=Qd_{\max}+(md_{\max}+2)\), the integer \(Q\) divides
\(N=md_{\max}+2\).  Write \(N=LQ\).  Then
\[
d_{\max}=\frac{LQ-2}{m},\qquad d_{\max}M+1=bQ.
\]
Multiplying the second equality by \(m\) and using \(Q=2M-m\), we obtain
\(
ML-bm=1\). 
It follows that \((b,L)=1\), and
\(L_G=N/(N,Md_{\max}+1)=L\).  Moreover \(LR\equiv1\pmod m\).  Hence \((L,m)=1\), so the
exceptional \(L_G=2\) alternative in Theorem \ref{theo-existence-dct} cannot
occur because \(m\) is even.  Therefore the existence of a \(d_{\max}\)-cluster tilting module is equivalent to the divisibility condition \(L\mid (m-1,d_{\max}+2)\).   Put
\(s=\frac{m-1}{L}\). 
Since \(Ls\equiv-1\pmod m\), the congruence \(LR\equiv1\pmod m\) gives
\(R\equiv -s\pmod m\).  Write \(R=am-s\).  Then
\(
Q=(2a+1)m-2s\), 
and hence
\[
d_{\max}=\frac{LQ-2}{m}=(2a+1)L-2=\frac{2u(m-1)^2-2s}{sm}.
\]
This also shows that $L\mid m-1$ implies that $L\mid d_{\max}+2=(2a+1)L$. Thus the divisibility condition \(L\mid (m-1,d_{\max}+2)\) is equivalent to \(L\mid m-1\). In this case, \(s\mid m-1\),  since \(m\) is even, this
\(s\) is odd.  Thus (A2) is proved.  
\end{proof}

The same method gives further examples in other representation-finite
self-injective types.  Once the rigidity-degree formulae in \cite{HuYinRigidityDegrees} determine
the maximal rigidity degree \(d_{\max}\), Theorem
\ref{theo-existence-dct} can be applied to decide whether a
\(d_{\max}\)-cluster tilting module exists.  We do not attempt a complete
classification in type \(\mathbb D\) here; the next corollary records several
closed families obtained in this way.

\begin{Coro}
\label{coro-rigdim-type-D-closed-families}
Let \(A\) be an indecomposable non-semisimple representation-finite
self-injective algebra of tree class \(\mathbb D_{m+1}\), where \(m\geq3\).
Write \(\Gamma_A=\mathbb Z\mathbb D_{m+1}/\langle\tau^R\psi\rangle\), where
\(\psi\) is induced by a graph automorphism.  We record the following closed
families for which \(A\) has \(d_{\max}\)-cluster tilting modules; hence
\(\rigdim A=d_{\max}+2\).
\begin{enumerate}[label={\rm(D\arabic*)}]
\item If \(\psi=\id\) and \(R=am+1\) with \(a\geq1\), then
\(d_{\max}=a\) when \(a(m+1)\) is even, and
\(d_{\max}=a(m+1)+1\) when \(a(m+1)\) is odd.
\item If \(o(\psi)=2\) and \(R=am+1\) with \(a\geq1\), then
\(d_{\max}=a\) when \(a(m+1)\) is odd, and
\(d_{\max}=a(m+1)+1\) when \(a(m+1)\) is even.
\item If \(\psi=\id\) and \(R=2m\), then \(d_{\max}=1\).
\end{enumerate}
\end{Coro}

\begin{proof}
Let \(d_1=\rd(1)\) be the
rigidity degree of the far end of the long arm, and let
\(d_m=\rd(m_+)=\rd(m_-)\) be the rigidity degree of the two branch vertices.
For \(R=am+1\), it follows from the formulae in \cite[Theorem 5.1]{HuYinRigidityDegrees} that
\(d_1=a\), while \(d_m=a\) if \(a+R+o(\psi)\) is even and
\(d_m=a(m+1)+1\) if \(a+R+o(\psi)\) is odd.  This gives the values
of \(d_{\max}=\max\{d_1,d_m\}\) in (D1) and (D2).

If \(d_{\max}=a\), then the type \(\mathbb D_{m+1}\) cluster-category
parameter is \(N=md_{\max}+1=am+1=R\), so \(L_G=1\).  Hence Theorem
\ref{theo-existence-dct} applies.  If \(d_{\max}=a(m+1)+1\), then
\(N=m(a(m+1)+1)+1=(m+1)(am+1)=(m+1)R\).  Hence \(L_G=m+1\), and in the
notation of Theorem \ref{theo-existence-dct} one has \(q=RL_G/N=1\).  For
\(\psi=\id\), this second case is exactly the parity condition that
\(a(m+1)\) is odd; hence \(n=m+1\) is odd and \(d_{\max}\) is even, so the
type \(\mathbb D\) row for \(\tau^R\) only asks for \(L_G\mid n\).  For
\(o(\psi)=2\), the second case is exactly the parity condition that
\(a(m+1)\) is even.  If \(n=m+1\) is even, then \(n/L_G=1\) is odd; if \(n\)
is odd, then \(q=1\) is odd.  Thus the corresponding order-two row of Theorem
\ref{theo-existence-dct} applies in either subcase.

Finally assume \(\psi=\id\) and \(R=2m\).  This is the case \(R=am\) with
\(a=2\).  The formula in \cite{HuYinRigidityDegrees} gives \(d_1=d_m=1\), so \(d_{\max}=1\).  Now
\(N=m+1\).  If \(m\) is even, then \(L_G=m+1\), \(q=2m\) is even, and the
\((n,d)=(\text{odd},\text{odd})\) row for \(\tau^R\) applies.  If \(m\) is
odd, then \(L_G=(m+1)/2\), so \(n/L_G=2\), and the
\((n,d)=(\text{even},\text{odd})\) row applies.  This proves (D3).  The
rigidity dimension statement follows from the general discussion at the
beginning of this section.
\end{proof}

\begin{Example}
Let \(A\) be a representation-finite self-injective algebra with
\(\Gamma_A=\mathbb Z\mathbb A_4/\langle\tau^{14}\rangle\).  Here
\(m=5\) and \(14=3m-1\), so Corollary
\ref{coro-rigdim-type-A-closed-families} gives \(d_{\max}=22\).  The induced
order in \(\mathcal C_{22}(k\mathbb A_4)\) is \(4\), and
\(4\mid\gcd(4,24)\).  Thus \(A\) has a maximal \(22\)-orthogonal
module and \(\rigdim A=24\).
\end{Example}

\begin{Example}
Let \(A\) be a representation-finite self-injective algebra with
\(\Gamma_A=\mathbb Z\mathbb D_6/\langle\tau^{36}\rangle\).  In the notation of
Corollary \ref{coro-rigdim-type-D-closed-families}, \(m=5\) and
\(36=7m+1\).  Since \(7(m+1)\) is even, Corollary
\ref{coro-rigdim-type-D-closed-families}(D1) gives \(d_{\max}=7\).  Hence
\(\rigdim A=9\).
\end{Example}

\section*{Acknowledgements}

 The authors are partially supported by
Beijing Natural Science Foundation (1252011). The authors are grateful to Professors Changjian Fu, Pin Liu, Yu Zhou and Bin Zhu
for many helpful discussions.

\medskip 
Qi Bin, School of Mathematical Sciences, Beijing Normal University, Beijing 100875, China. Email: binqi@mail.bnu.edu.cn

\medskip 
Yifu Han, School of Mathematical Sciences, Beijing Normal University, Beijing 100875, China. Email: 
xiaofuhan@mail.bnu.edu.cn

\medskip 
Wei Hu, School of Mathematical Sciences, Beijing Normal University, Beijing 100875, China. Email: huwei@bnu.edu.cn

\end{document}